\documentclass[letterpaper]{amsart}

\usepackage[utf8]{inputenc}
\usepackage[T1]{fontenc}
\usepackage{lmodern}
\usepackage{MnSymbol}
\usepackage{mathtools}
\usepackage[pdfusetitle]{hyperref}

\def\arxiv#1{\href{http://arxiv.org/abs/#1}{\texttt{arXiv:#1}}}
\def\doi#1{\href{http://dx.doi.org/#1}{doi:#1}}
\usepackage{tikz-cd}

\let\phi\varphi

\def\N{\mathbb N}
\def\Z{\mathbb Z}

\def\R{\mathbb R}
\def\C{\mathbb C}

\def\pair#1{\langle#1\rangle}

\def\ttt{\mathfrak t}
\def\ggg{\mathfrak g}

\def\lll{\mathfrak l}

\def\ppp{\mathfrak p}
\def\qqq{\mathfrak q}

\def\CT{C_{T}}

\def\CTc{\Omega_{T,c}}

\def\Hc{H_{c}}
\def\HT{H_{T}}
\def\HG{H_{G}}

\def\HTc{H_{T,c}}
\def\HGc{H_{G,c}}

\def\hHc{H^{c}}
\def\hHT{H^{T\!}}
\def\hHG{H^{G\!}}

\def\hHTc{H^{T,c}}
\def\hHGc{H^{G,c}}
\def\RG{R_{G}}
\def\RT{R_{T}}
\def\RK{R_{K}}
\def\RL{R_{L}}
\def\AB#1#2{AB_{#1}^{\,#2}}

\def\barAB#1#2{{\smash{\overline{AB\mathstrut}}\mathstrut}_{\mkern-1.5mu #1}^{\,#2}}

\def\FF{\mathcal F}

\let\MFstar\thinstar 
\def\Gstar{G^{\MFstar}}

\def\WWW{\mathcal{W}}
\def\Wstar{\WWW^{\MFstar}}

\def\PD{PD}
\def\tZ{\tilde Z}
\def\kktilde{\tilde\R}

\def\Xhat{\hat{X}}
\def\Yhat{\hat{Y}}
\def\Zhat{\hat{Z}}

\def\Cc{\Omega_{c}}
\def\CTc{\Omega_{T,c}}

\def\hCG{C^{G}}
\def\hCT{C^{T}}
\def\dirlim{\mathop{\underrightarrow\lim}}

\def\CarG{C_{G}}
\def\CarT{C_{T}}
\def\CarK{C_{K}}

\def\HHH{\mathcal{H}}

\let\RR\RT
\def\RRa{\R[\ttt^{*}_{\alpha}]}
\let\kk\R
\def\CLc{\Omega_{L,c}}
\def\HLc{H_{L,c}}

\def\hHLc{H^{L,c}}
\def\SS{S}

\DeclareMathOperator{\Hom}{Hom}
\DeclareMathOperator{\Ext}{Ext}
\DeclareMathOperator{\rank}{rank}
\DeclareMathOperator{\Ad}{Ad}
\DeclareMathOperator{\Gr}{Gr}

\DeclareMathOperator{\depth}{depth}
\DeclareMathOperator{\height}{ht}

\def\ie{\emph{i.\,e.}}
\def\cf{\emph{cf.}}

\theoremstyle{plain}
\newtheorem{theorem}{Theorem}[section]
\newtheorem{proposition}[theorem]{Proposition}
\newtheorem{lemma}[theorem]{Lemma}
\newtheorem{corollary}[theorem]{Corollary}

\theoremstyle{definition}

\newtheorem{example}[theorem]{Example}

\newtheorem{remark}[theorem]{Remark}
\newtheorem{assumption}[theorem]{Assumption}

\theoremstyle{remark}
\newtheorem*{acknowledgements}{Acknowledgements}

\numberwithin{equation}{section}

\begin{document}

\title[Syzygies in equivariant cohomology]%
  {Syzygies in equivariant cohomology\\for non-abelian Lie groups}
\author{Matthias Franz}
\thanks{The author was supported by an NSERC Discovery Grant.}
\address{Department of Mathematics, University of Western Ontario, London, Ont.\ N6A\;5B7, Canada}
\email{mfranz@uwo.ca}


\begin{abstract}
  We extend the work of Allday--Franz--Puppe on syzygies in equivariant cohomology
  from tori to arbitrary compact connected Lie groups~\(G\).
  In particular, we show that for a compact orientable \(G\)-manifold~\(X\) the analogue of the
  Chang--Skjelbred sequence is exact if and only if the equivariant cohomology of~\(X\) is reflexive,
  if and only if the equivariant Poincaré pairing for~\(X\) is perfect.
  Along the way we establish that the equivariant cohomology modules
  arising from the orbit filtration of~\(X\) are Cohen--Macaulay.
  We allow singular spaces and introduce a Cartan model for their equivariant cohomology.
  We also develop a criterion for the finiteness of the number of infinitesimal
  orbit types of a \(G\)-manifold.
\end{abstract}

\hypersetup{pdftitle={Syzygies in equivariant cohomology for non-abelian Lie groups}}

\maketitle

\section{Introduction}

Let \(R\) be a polynomial ring in \(r\)~indeterminates,
and let \(M\) be a finitely generated module over~\(R\).
Then \(M\) is called a \emph{\(j\)-th syzygy} if there is
an exact sequence
\begin{equation}
  0 \to M \to F_{1}\to \dots \to F_{j}
\end{equation}
with finitely generated free \(R\)-modules~\(F_{1}\),~\dots,~\(F_{j}\).
The first syzygies are exactly the torsion-free modules,
the second syzygies the reflexive ones, and the
\(r\)-th syzygies are free. In this sense, syzygies
interpolate between torsion-freeness and freeness.
Allday, Puppe and the author initiated the study of syzygies
in the context of torus-equivariant cohomology \cite{AlldayFranzPuppe:orbits1},~\cite{AlldayFranzPuppe:orbits4}.
To illustrate their results, 
let us focus on the second syzygies.
We use Alexander--Spanier cohomology with real coefficients.

Let \(T\cong(S^{1})^{r}\) be a torus, and let \(X\) be a \(T\)-space
with finite Betti sum
and satisfying some other mild assumptions.
Then \(R_{T}=H^{*}(BT)\) is a polynomial ring in \(r\)~indeterminates of degree~\(2\), and
the equivariant cohomology~\(\HT^{*}(X)\) of~\(X\) is an \(\RT\)-module.
Allday--Franz--Puppe showed that the Chang--Skjelbred sequence
\begin{equation}
  \label{eq:cs-intro}
  0 \to \HT^{*}(X) \to \HT^{*}(X^{T}) \to \HT^{*+1}(X_{1,T},X^{T})
\end{equation}
is exact if and only if \(\HT^{*}(X)\) is a reflexive \(\RT\)-module.
Here \(X^{T}\) denotes the fixed point set, and \(X_{1,T}\) is the union
of~\(X^{T}\) and all \(1\)-dimensional orbits.
The sequence~\eqref{eq:cs-intro} often permits an efficient computation of~\(\HT^{*}(X)\). 
Moreover, if \(X\) is a Poincaré duality space, then the reflexivity
of~\(\HT^{*}(X)\) is also equivalent
to the perfection of the equivariant Poincaré pairing
\begin{equation}
  \HT^{*}(X) \times \HT^{*}(X) \to \RT,\
  \quad
  (\alpha,\beta)\mapsto \pair{\alpha\cup\beta,o_{T}},
\end{equation}
which is the composition of the cup product and
evaluation on an equivariant orientation~\(o_{T}\) of~\(X\).

The main purpose of the present paper is to extend the results
of~\cite{AlldayFranzPuppe:orbits1} and~\cite{AlldayFranzPuppe:orbits4}
to actions of arbitrary compact connected Lie groups.
It turns out that essentially all results
carry over to this more general setting. We achieve this by combining
the techniques of Allday--Franz--Puppe
with those of Goertsches--Rollenske~\cite{GoertschesRollenske:2011},
whose study of Cohen--Macaulay actions
gave a first hint at the possibility of such an extension.
Let us describe our results in more detail.

Let \(G\) be a compact connected Lie group with maximal torus~\(T\)
and corresponding Weyl group~\(W\).
Then \(r\) is the rank of~\(G\), and \(\RG=H^{*}(BG)\) is again
a polynomial ring in \(r\)~indeterminates of even degrees.
Let \(X\) be a \(G\)-space;
we always assume that it admits an equivariant closed embedding
into some smooth \(G\)-manifold or \(G\)-orbifold.
This allows us to define a Cartan model
for the equivariant cohomology~\(\HG^{*}(X)\) of~\(X\) and also for
a suitably defined \emph{equivariant homology}~\(\hHG_{*}(X)\).
The latter is \emph{not} the homology of the Borel construction of~\(X\);
it is related to equivariant cohomology
via a universal coefficient spectral sequence (Proposition~\ref{thm:UCT})
and for manifolds and locally orientable orbifolds also via
equivariant Poincaré duality (Proposition~\ref{thm:PD-equiv}).
Here ``locally orientable'' means that \(X\) is a rational homology manifold,
see Section~\ref{sec:PD-equiv} for the precise definition.
For most results, we also assume that
only finitely many infinitesimal orbit types occur in~\(X\).
(The infinitesimal orbit type of~\(x\in X\) is the
orbit of the Lie algebra~\(\ggg_{x}\) of the stabilizer~\(G_{x}\) under the adjoint action of \(G\).)
This condition is often redundant, see below.

The \emph{\(G\)-orbit filtration} of~\(X\) is defined by
\begin{equation}
  X_{i,G} = \bigl\{\, x\in X \bigm| \rank G_{x} \ge r-i \,\bigr\}
\end{equation}
for~\(-1\le i\le r\).
Then \(X_{-1,G}=\emptyset\), \(X_{0,G}\) is the maximal-rank stratum
and \(X_{r,G}=X\). All~\(X_{i,G}\) are \(G\)-stable and closed in~\(X\).
Note that if \(G=T\), then \(X_{0,T}=X^{T}\). 
The maximal-rank stratum
plays the role of the fixed point set in the non-abelian context.

Based on the work of Goertsches--Rollenske, we show:

\begin{proposition}
  \label{thm:incl-Y-X-intro}
  Let \(X\) be a \(G\)-space.
  For~\(0\le i\le r\),
  there is
  an isomorphism of \(\RG\)-modules
  \begin{equation*}
    \HG^{*}(X_{i,G},X_{i-1,G}) = \HT^{*}(X_{i,T},X_{i,T}\cap X_{i-1,G})^{W}.
  \end{equation*}
\end{proposition}

An immediate consequence is the following:

\begin{corollary}
  \label{thm:HGXi-CM-intro}
  Let \(X\) be a \(G\)-space with only finitely many infinitesimal orbit types and
  such that \(H^{*}(X)\) is finite-dimensional.
  The \(\RG\)-module
  \(\HG^{*}(X_{i,G},X_{i-1,G})\)
  is zero or Cohen--Macau\-lay of dimension~\(r-i\) for~\(0\le i\le r\).
\end{corollary}

For~\(i=0\), both results were obtained independently
by Baird~\cite[Sec.~3]{Baird:2007} 
and Goertsches--Rollenske~\cite[Sec.~3]{GoertschesRollenske:2011}. 
The latter authors proved Corollary~\ref{thm:HGXi-CM-intro}
also for compact \(G\)-manifolds~\(X\)
such that \(X_{i-1}=\emptyset\) \cite[Cor.~4.3]{GoertschesRollenske:2011}.

As for torus actions, we consider the \emph{Atiyah--Bredon sequence}
\begin{multline}
  \label{eq:AB-intro}
  0\longrightarrow \HG^{*}(X) \stackrel{\iota^{*}}\longrightarrow \HG^{*}(X_{0,G})
  \stackrel{\delta_{0}}\longrightarrow \HG^{*+1}(X_{1,G},X_{0,G})\stackrel{\delta_{1}}\longrightarrow \cdots \\
  \stackrel{\delta_{r-2}}\longrightarrow \HG^{*+r-1}(X_{r-1,G},X_{r-2,G})
  \stackrel{\delta_{r-1}}\longrightarrow \HG^{*+r}(X_{r,G},X_{r-1,G})
  \longrightarrow 0.
\end{multline}
Here \(\iota^{*}\) is induced by the inclusion~\(\iota\colon X_{0,G}\hookrightarrow X\),
and \(\delta_{i}\) for~\(i\ge0\) is
the connecting homomorphism in the long exact sequence
for the triple~\((X_{i+1,G},X_{i,G},X_{i-1,G})\).
The \emph{Atiyah--Bredon complex}~\(\AB{G}{*}(X)\)
with~\(\AB{G}{i}(X)=\HG^{*+i}(X_{i,G},X_{i-1,G})\)
is obtained from~\eqref{eq:AB-intro} by dropping the leading term.

\begin{theorem}
  \label{thm:Ext=HAB-intro}
  Let \(X\) be a \(G\)-space with only finitely many infinitesimal orbit types and
  such that \(H^{*}(X)\) is finite-dimensional.
  For any~\(j\ge0\),
  the \(j\)-th cohomology of the Atiyah--Bredon complex is
  \begin{equation*}
    H^{j}(\AB{G}{*}(X)) = \Ext_{\RG}^{j}(\hHG_{*}(X),\RG).
  \end{equation*}
\end{theorem}

This allows us to characterize syzygies in equivariant cohomology:

\begin{theorem}
  \label{thm:partial-exact-intro}
  Let \(X\) be a \(G\)-space with only finitely many infinitesimal orbit types and
  such that \(H^{*}(X)\) is finite-dimensional,
  and let \(1\le j\le r\). Then
  \(\HG^{*}(X)\) is a \(j\)-th syzygy over~\(\RG\)
  if and only if the part
  \begin{equation*}
    0\to \HG^{*}(X) \to \HG^{*}(X_{0,G}) \to \cdots \to \HG^{*+j-1}(X_{j-1,G},X_{j-2,G})
  \end{equation*}
  of the Atiyah--Bredon sequence is exact.
\end{theorem}

In the torus case this result is due to
Allday--Franz--Puppe~\cite[Thm.~5.7]{AlldayFranzPuppe:orbits1},~\cite[Thm.~4.8]{AlldayFranzPuppe:orbits4}.
For general~\(G\) and compact~\(X\),
the case~\(j=1\) has been established by 
Allday (unpublished, based on the case~\(G=SU(2)\) solved by Chen) and
Goertsches--Rollenske~\cite[Thm.~3.9]{GoertschesRollenske:2011}.
Whether it extends to higher syzygies in the non-abelian setting has been an open question,
see~\cite[Question~4.7]{GoertschesRollenske:2011}.

\goodbreak

In the important special case~\(j=2\) we get:

\begin{corollary}
  Let \(X\) be a compact orientable \(G\)-manifold or \(G\)-orbifold.
  The following are equivalent:
  \begin{enumerate}
  \item The \(\RG\)-module~\(\HG^{*}(X)\) is reflexive.
  \item The sequence 
    \begin{equation*}
      0 \to \HG^{*}(X) \to \HG^{*}(X_{0,G}) \to \HG^{*+1}(X_{1,G},X_{0,G})
    \end{equation*}
    is exact.
  \item The equivariant Poincaré pairing \(\HG^{*}(X) \times \HG^{*}(X) \to \RG\) is perfect.
  \end{enumerate}
\end{corollary}

That the freeness of~\(\HG^{*}(X)\) over~\(\RG\) implies the exactness
of the non-abelian Chang--Skjelbred sequence is due
to Brion~\cite[Thm.~9]{Brion:1997} for compact multiplicity-free \(G\)-spaces
and to Goertsches--Mare~\cite[Thm.~2.2]{GoertschesMare:2014}
for compact \(G\)-manifolds.
The perfection of the equivariant Poincaré pairing was shown by
Ginzburg~\cite[Cor.~3.9]{Ginzburg:1987} 
and Brion~\cite[Prop.~1]{Brion:2000} 
again under the assumption that \(\HG^{*}(X)\) is free over~\(\RG\).

\smallskip

Let \(X\) be a \(G\)-space.
In~\cite{AlldayFranzPuppe:orbits1} and~\cite{AlldayFranzPuppe:orbits4},
\(T\)-spaces were assumed to have finite Betti sum and also finitely
many infinitesimal orbit types, and as mentioned above,
we often require the same for \(G\)-spaces.
As another application of equivariant homology,
we show that the latter condition is redundant for manifolds and
locally orientable orbifolds.

\begin{theorem}
  \label{thm:finite-inf-orbits-intro}
  Let \(X\) be a \(G\)-manifold or locally orientable \(G\)-orbifold.
  If \(H^{*}(X)\) is finite-dimensional, then
  only finitely many infinitesimal orbit types occur in~\(X\).
\end{theorem}

If \(X\) is compact, it follows easily
from the differentiable slice theorem
that there are actually
only finitely many orbits types in~\(X\),
\cf~\cite[Prop.~VIII.3.13]{Borel:1960}.
(For the orbifold version of the differentiable slice theorem,
see~\cite[Prop.~2.3]{LermanTolman:1997}.)
By a result of Mann~\cite[Thm.~3.5]{Mann:1962},
the same conclusion holds if
\(X\) is an orientable manifold (or cohomology manifold) with finitely generated \emph{integral} cohomology.
While Theorem~\ref{thm:finite-inf-orbits-intro} is in the same spirit as Mann's result, the proof is very different.

\smallskip

The paper is organized as follows:
After discussing Cohen--Macaulay modules, syzygies and the algebraic aspects of Cartan models in Section~\ref{sec:prelim},
we define equivariant de~Rham homology and cohomology for possibly singular spaces in Section~\ref{sec:equiv-cohom-hom}.
In Section~\ref{sec:res-ind} we discuss the behaviour of syzygies under restriction to and induction from
a maximal torus. To prepare for the proof of the main results, we study
spaces with isotropy groups of constant rank in Section~\ref{sec:constant}.
The results of~\cite{AlldayFranzPuppe:orbits1} and~\cite{AlldayFranzPuppe:orbits4}
are generalized in Section~\ref{sec:filtration}
where we also introduce the notion of a Cohen--Macaulay filtration of a \(G\)-space.
In the final Section~\ref{sec:finite} we
prove Theorem~\ref{thm:finite-inf-orbits-intro}.

\begin{acknowledgements}
  I would like to thank Chris Allday, Andrey Minchenko,
  Volker Puppe, Reyer Sjamaar and Andrzej Weber for helpful discussions.
\end{acknowledgements}

\section{Algebraic preliminaries}
\label{sec:prelim}

\subsection{Notation and standing assumptions}

Throughout this paper, the letter~\(G\) denotes a compact connected Lie group of rank~\(r\)
with maximal torus~\(T\cong(S^{1})^{r}\) and corresponding Weyl group~\(W=N_{G}(T)/T\).
The Lie algebra of~\(G\) is written as~\(\ggg\) and its dual as~\(\ggg^{*}\).
We adopt this naming scheme for all Lie groups.

All manifolds and orbifolds are assumed to be paracompact.
(See~\cite[Sec.~1.1]{AdemLeidaRuan:2007} for the definition of an orbifold.)
All (co)homology is taken with real coefficients.

We adopt a cohomological grading for all complexes, so that differentials always have degree~\(+1\).
The degree of an element~\(c\) of a complex is denoted by~\(|c|\).
Homological complexes are turned into cohomological ones by grading them negatively.
For example, elements in the \(n\)-th homology group~\(H_{n}(X)\) of a space~\(X\) have degree~\(-n\).
A quasi-isomorphism is a map of complexes inducing an isomorphism in cohomology.

Let \(A\)~and~\(B\) be complexes.
The \(n\)-th degree of the complex~\(\Hom_{\R}(A,B)\) consists
of all linear maps~\(f\colon A\to B\) that raise degrees by~\(n\in\Z\). The differential is
defined by
\begin{equation}
  d(f) = d_{B}\circ f- (-1)^{n}f\circ d_{A}.
\end{equation}
This generalizes to differential graded~(dg) modules over some dg~algebra.
The complex~\(A^{\vee}=\Hom_{\R}(A,\R)\) is the \emph{dual complex} of~\(A\).
Moreover, for~\(k\in\Z\), the \(n\)-th degree of the \emph{shifted complex}~\(A[k]\) is equal to~\(A^{n+k}\);
the differential of~\(A[k]\) is \((-1)^{k}d_{A}\). We also write \(H^{*+k}(X)\) instead of~\(H^{*}(X)[k]\).

Unless specified otherwise,
all tensor products are over~\(\R\).

\subsection{Cohen--Macaulay modules}

Let \(R\) be a polynomial ring over a field in \(r\)~indeterminates.
We assume that \(R\) is graded by assigning a positive degree to each indeterminate.
Recall that a graded \(R\)-module~\(M\ne0\) is called \emph{Cohen--Macaulay}
if it is 
finitely generated and
\begin{equation}
  \depth M = \dim M.
\end{equation}

\begin{proposition}
  \label{thm:prop-CM}
  Let \(M\) be a finitely generated graded \(R\)-module, and let \(d\in\N\).
  \begin{enumerate}
  \item \(M\) is Cohen--Macaulay of dimension~\(d\) if and only if
    \begin{equation*}
      \Ext_{R}^{i}(M,R) \ne 0
      \quad\Leftrightarrow\quad
      i = r-d.
    \end{equation*}
  \item In this case, \(\Ext_{R}^{r-d}(M,R)\) is again Cohen--Macaulay
    of dimension~\(d\).
  \end{enumerate}
\end{proposition}

\begin{proof}
  See~\cite[Prop.~A1.16]{Eisenbud:2005} for the first part.
  A version of the second part for local rings
  can be found in~\cite[Prop.~3.3.3\,(b)]{BrunsHerzog:1998};
  our claim follows from this by localizing \(\Ext_{R}^{*}(\Ext_{R}^{i}(M,R)\))
  at all primes~\(\ppp\lhd R\),
  noting that \(R_{\ppp}\) is a Cohen--Macaulay local ring.  
\end{proof}

\subsection{Syzygies}

Let \(R\) be a polynomial ring in \(r\)~indeterminates over a field.
The definition of a syzygy has been given in the introduction.
We observe that any finitely generated \(R\)-module is a zeroth syzygy.
From~\cite[App.~E]{BrunsVetter:1988} we note:

\begin{proposition}
  \label{thm:syzygy}
  Let \(M\) be a finitely generated \(R\)-module.
  The following are equivalent for any~\(0\le j\le r\):
  \begin{enumerate}
  \item \label{thm:syzygy-1}
    \(M\) is a \(j\)-th syzygy.
  \item \label{thm:syzygy-2}
    Any regular sequence~\(f_{1}\),~\dots,~\(f_{j}\in R\) is \(M\)-regular.
  \item \label{thm:syzygy-3}
    \(\depth_{R_{\ppp}} M_{\ppp}\ge\min(j,\dim R_{\ppp})\) for any prime ideal~\(\ppp\lhd R\).
  \end{enumerate}
\end{proposition}

\subsection{Cartan models}
\label{sec:cartan-models}

Recall that a \(\Gstar\)-module
is a complex~\(A\) of \(G\)-modules with operations~\(L_{\xi}\) (of degree~\(0\)) and~\(\iota_{\xi}\) (of degree~\(-1\)) for~\(\xi\in\ggg\)
satisfying the same relations as the Lie derivative and the contraction operator for differential forms;
see~\cite[Def.~2.3.1]{GuilleminSternberg:1999} for details.

The Cartan model of a \(\Gstar\)-module~\(A\) is denoted by
\begin{equation}
  \label{eq:def-Cartan-model}
  \CarG^{*}(A) = (\R[\ggg^{*}]\otimes A)^{G}
\end{equation}
where \(\R[\ggg^{*}]\) denotes the real-valued polynomials on~\(\ggg\),
graded by twice the polynomial degree.
The differential is
\begin{equation}
  d(f\otimes a) = f\otimes d a - \sum_{k=1}^{r}f x_{k}\otimes\iota_{\xi_{k}}a \; ;
\end{equation}
here \((\xi_{k})\) is a basis for~\(\ggg\) with dual basis~\((x_{k})\).
The cohomology of~\(\CarG^{*}(A)\) is the \emph{equivariant cohomology} of~\(A\) and denoted by~\(\HG^{*}(A)\).
The coefficient field~\(\R\) is a \(\Gstar\)-module 
with trivial operations;
we set \(\RG=\HG^{*}(\R)=\CarG^{*}(\R)=\R[\ggg^{*}]^{G}\).
The Cartan model~\eqref{eq:def-Cartan-model} 
is a dg~\(\RG\)-module, hence \(\HG^{*}(A)\) is an \(\RG\)-module.

\begin{remark}
  If \(A\) is a \(\Gstar\)-algebra in the sense of~\cite[Def.~2.3.1]{GuilleminSternberg:1999},
  then \(\CarG^{*}(A)\) is a dg~\(\RG\)-algebra by componentwise multiplication,
  and \(\HG^{*}(A)\) is an \(\RG\)-algebra. However, most of the time
  we will not be concerned with the multiplicative structure
  in equivariant cohomology. The only exceptions are Section~\ref{sec:PD-equiv},
  where we use the cup product in the equivariant de~Rham cohomology
  of an orientable manifold or orbifold to establish equivariant Poincaré duality,
  and Section~\ref{sec:partial}, where we define the equivariant Poincaré pairing.
\end{remark}

\begin{lemma}
  \label{thm:Car-serre}
  Let \(A\) be a bounded below \(\Gstar\)-module. 
  There is a first-quadrant spectral sequence, natural in~\(A\), with
  \begin{equation*}
    E_{1} = \RG\otimes H^{*}(A) \;\Rightarrow\; H_{G}^{*}(A).
  \end{equation*}
\end{lemma}

\begin{proof}
  See~\cite[Thm.~6.5.2]{GuilleminSternberg:1999}.
\end{proof}

Let \(A\) be a \(\Gstar\)-module. 
The inclusion~\(K\hookrightarrow G\) of a closed subgroup induces
a canonical restriction map~\(\CarG^{*}(A)\to\CarK^{*}(A)\), which is a 
chain map compatible with the restriction map~\(\RG\to\RK\). For~\(K=1\) we get
a map~\(\CarG^{*}(A)\to A\), which induces the restriction map 
\begin{equation}
  \label{eq:Car-restriction}
  \HG^{*}(A)\to H^{*}(A).
\end{equation}

\smallskip

Let \(\WWW=\R[\ggg^{*}]\otimes\bigwedge\ggg^{*}\) be the Weil algebra of~\(\ggg^{*}\),
\cf~\cite[Sec.~3.2]{GuilleminSternberg:1999},
and let \(A\) be a \(\Wstar\)-module, see~\cite[Def.~3.4.1]{GuilleminSternberg:1999}
for the definition. For example, if \(G\) acts locally freely on a manifold~\(X\),
then the de~Rham complex~\(\Omega^{*}(X)\)
is a \(\Wstar\)-module, \cf~Lemma~\ref{thm:cond-C} below.
There is a canonical isomorphism
(with some twisted differential on the right-hand side)
\begin{equation}
  \label{eq:A-iso-connection}
  A = \bigwedge \ggg^{*}\otimes A_{\rm hor}
\end{equation}
where \(A_{\rm hor}\subset A\) denotes the subcomplex of horizontal elements,
\cf~\cite[Thm.~3.4.1]{GuilleminSternberg:1999}.
The basic subcomplex~\(A_{\rm bas}=(A_{\rm hor})^{G}\) is canonically
a dg~module over 
\(\WWW_{\rm bas}=\RG\).

\begin{lemma}
  \label{thm:retract-CarGA-Abas}
  Let \(A\) be a \(\Wstar\)-module.
  As a complex, \(A_{\rm bas}\) is a deformation retract
  of~\(\CarG^{*}(A)\). There is a natural deformation retraction
  \( 
    \CarG^{*}(A)\to A_{\rm bas}
  \) 
  which is a morphism of dg~\(\RG\)-modules.
  In particular, there is a natural isomorphism
  of \(\RG\)-modules 
  \begin{equation*}
    \HG^{*}(A) = H^{*}(A_{\rm bas}).
  \end{equation*}
\end{lemma}

\begin{proof}
  See in particular \cite{Nicolaescu:1999}, or~\cite[Ch.~5]{GuilleminSternberg:1999}.
\end{proof}

\smallskip

The dual complex~\(A^{\vee}\)
of a \(\Gstar\)-module~\(A\) is again a \(\Gstar\)-module
via the assignments
\begin{align}
  \pair{g\phi,a} &= \pair{\phi,g^{-1}a}, &
  \pair{L_{\xi}\phi,a} &= -\pair{\phi,L_{\xi} a}, &
  \pair{\iota_{\xi}\phi,a} &= -(-1)^{|\phi|}\pair{\phi,\iota_{\xi}a}
\end{align}
for~\(\phi\in A^{\vee}\),~\(a\in A\),~\(g\in G\) and~\(\xi\in\ggg\).
Here \(\pair{\phi,a}\) denotes the pairing between~\(A^{\vee}\) and~\(A\).
Moreover, if \(A\) is a \(\Wstar\)-module, then so is \(A^{\vee}\), where the
\(\WWW\)-module structure is defined by
\begin{equation}
  \pair{w\phi,a} = (-1)^{|w|\,|\phi|}\pair{\phi,w a}
\end{equation}
for~\(\phi\in A^{\vee}\),~\(a\in A\) and~\(w\in\WWW\).

\begin{lemma}
  \label{thm:Adual-Wstar}
  For any \(\Wstar\)-module~\(A\)
  there is an isomorphism of dg~\(\RG\)-modules
  \begin{equation*}
    (A^{\vee})_{\rm bas} \cong (A_{\rm bas})^{\vee}\,[-\dim G].
  \end{equation*}
\end{lemma}

\begin{proof}
  Dualizing \eqref{eq:A-iso-connection}, we get an isomorphism of \(\Wstar\)-modules
  \begin{equation}
    \label{eq:iso-Avee-ggg}
    A^{\vee} = \bigwedge \ggg \otimes (A_{\rm hor})^{\vee}
  \end{equation}
  where \(\bigwedge\ggg^{*}\subset\Wstar\) acts on the right-hand side
  through the pairing with~\(\bigwedge \ggg\),
  and the operators~\(\iota_{\xi}\) act on~\(\bigwedge \ggg\) by exterior multiplication.
  Hence
  \begin{equation}
    (A^{\vee})_{\rm hor} = \bigwedge\nolimits^{\dim G}\ggg\otimes(A_{\rm hor})^{\vee} \cong (A_{\rm hor})^{\vee}[-\dim G],
  \end{equation}
  and the claim follows by taking \(G\)-invariants,
  given that \(\bigwedge\nolimits^{\dim G}\ggg\) is \(G\)-invariant.
\end{proof}

\subsection{Universal coefficient theorem}
\label{sec:uct}

Let \(A\) be a bounded \(\Gstar\)-module.
We want to relate the equivariant cohomology of~\(A\) to that of~\(A^{\vee}\).
The following result will be crucial for us.

\begin{lemma}[Kostant]
  \label{thm:Sg-free}
  There is a \(G\)-submodule~\(\HHH\subset\R[\ggg^{*}]\) such that
  the restricted multiplication map
  \begin{equation*}
    \RG\otimes\HHH = \R[\ggg^{*}]^{G}\otimes\HHH \to \R[\ggg^{*}]
  \end{equation*}
  is bijective, hence an isomorphism
  of \(G\)-modules and \(\RG\)-modules.
\end{lemma}

The \(\RG\)-action on~\(\RG\otimes\HHH\) is on the first factor only.
Note that we necessarily have \(\HHH^{G}=\R\), the constant polynomials.

\begin{proof}
  See~\cite[Thm.~0.2]{Kostant:1963};
  the submodule~\(\HHH\) is given by the harmonic polynomials on~\(\ggg\).
  The case of complex coefficients that Kostant considers
  is obtained from the real case by extension of scalars, so the result holds already with real coefficients.
\end{proof}

From Lemma~\ref{thm:Sg-free} we get isomorphisms of \(\RG\)-modules
\begin{gather}
  \label{eq:iso-CGcX}
  \CarG^{*}(A) = \bigl(\R[\ggg^{*}]\otimes A\bigr)^{G}
    = \RG\otimes(\HHH\otimes A)^{G}, \\
  \label{eq:iso-Hom-RG}
  \Hom_{\RG}(\CarG^{*}(A),\RG) = \Hom_{\R}((\HHH\otimes A)^{G},\RG).
\end{gather}

\begin{lemma}
  \label{thm:fg-free}
  As a dg~\(\RG\)-module, \(\CarG^{*}(A)\) is homotopy equivalent to~\(\RG\otimes H^{*}(A)\)
  with some twisted differential.
\end{lemma}

\begin{proof}
  Consider the complex
  \begin{equation}
    B = \R\otimes_{\RG}\CarG^{*}(A).
  \end{equation}
  Because \(\CarG^{*}(A)\) is free over~\(\RG\),
  it is \(\RG\)-homotopy equivalent to~\(\RG\otimes H^{*}(B)\)
  with some twisted differential
  by~\cite[Cor.~B.2.4]{AlldayPuppe:1993},
  \cf~also~\cite[Rem.~3.3]{AlldayFranzPuppe:orbits1}.
  We claim that \(H^{*}(B)\cong H^{*}(A)\).

  By~\eqref{eq:iso-CGcX} we have
  \( 
    B=(\HHH\otimes A)^{G}
  \) 
  with some differential.
  We filter this complex by degree in~\(A\). The \(E_{1}\)~page is
  \begin{equation}
    \label{eq:E1-HAG-HA}
    E_{1} = (\HHH\otimes H^{*}(A))^{G} = \HHH^{G}\otimes H^{*}(A) = H^{*}(A)
  \end{equation}
  because taking \(G\)-invariants and cohomology commute
  and \(G\) acts trivially in cohomology.
  The spectral sequence therefore degenerates, and
  \( 
    H^{*}(B) = H^{*}(A)
  \). 
\end{proof}

The canonical pairing~\(A^{\vee}\times A\to\R\), \((\phi,a)\mapsto\pair{\phi,a}\),
has the equivariant extension
\begin{equation}
  \CarG^{*}(A^{\vee})\times \CarG^{*}(A) \to \CarG^{*}(\R) = \RG,
\end{equation}
which is the restriction of the canonical pairing
\begin{equation}
  \R[\ggg^{*}]\otimes A^{\vee} \times \R[\ggg^{*}]\otimes A \to \R[\ggg^{*}],
  \qquad
  (f_{1}\otimes\phi,f_{2}\otimes a) \mapsto \pair{\phi,a} f_{1} f_{2}.
\end{equation}
Hence we get a morphism of dg~\(\RG\)-modules
\begin{equation}
  \label{eq:map-Car-Hom}
  \Phi\colon \CarG^{*}(A^{\vee}) 
  \to \Hom_{\RG}(\CarG^{*}(A),\RG).
\end{equation}
For~\(G=T\) a torus, it is an isomorphism
as one can see by comparing \eqref{eq:iso-CGcX} with~\eqref{eq:iso-Hom-RG}.
In general we have the following:

\begin{lemma}
  \label{thm:quiso-Car-Hom}
  The map~\(\Phi\)
  is a quasi-iso\-mor\-phism of dg~\(\RG\)-modules.
\end{lemma}

\begin{proof}
  We filter \eqref{eq:iso-CGcX} (with~\(A^{\vee}\) instead of~\(A\)) and~\eqref{eq:iso-Hom-RG} by degree in~\(\RG\);
  this is compatible with the map~\eqref{eq:map-Car-Hom}.
  The differentials~\(d_{0}\) on both \(E_{0}\)~pages come from the differential in~\(A\),
  hence the \(E_{1}\)~pages are
  \begin{align}
    \label{eq:E1-1}
    E_{1} &= \RG\otimes H^{*}\bigl( (\HHH\otimes A^{\vee})^{G}\bigr) = \RG\otimes H^{*}(A)^{\vee} \\
    \shortintertext{in the first case, and}
    \label{eq:E1-2}
    E_{1} &= \Hom_{\R}(H^{*}(A),\RG) = \RG\otimes H^{*}(A)^{\vee}
  \end{align}
  in the second.
  We claim that the map~\(E_{1}(\Phi)\) between them 
  is an isomorphism:
  Represent an element of~\eqref{eq:E1-1} by~\(f\otimes\phi\)
  where \(\phi\) is a \(G\)-invariant cycle in~\(A^{\vee}\) 
  and \(f\in\RG\). It maps again to~\(f\otimes[\phi]\) in~\eqref{eq:E1-2}, proving the claim.
  We therefore get an isomorphism between the \(E_{\infty}\)~pages.

  By assumption, \(\CarG^{*}(A^{\vee})\) is bounded below,
  so that the spectral sequence~\eqref{eq:E1-1} converges naively.
  The \(\RG\)-homotopy equivalence give by Lemma~\ref{thm:fg-free} induces one
  between~\eqref{eq:iso-Hom-RG} and
  \begin{equation}
    \label{eq:E1-3}
    \Hom_{\RG}(\RG\otimes H^{*}(A),\RG) = \Hom_{\R}(H^{*}(A),\RG),
  \end{equation}
  which is compatible with the filtration by \(\RG\)-degree.
  Moreover, the map between the \(E_{0}\)~pages of the associated spectral sequences
  comes from the
  induced homotopy equivalence between~\((\HHH\otimes A^{\vee})^{G}\) and~\(H^{*}(A)\),
  hence is itself a homotopy equivalence.
  From the \(E_{1}\)~pages on, the spectral sequence~\eqref{eq:E1-2} therefore
  coincides with the one for~\eqref{eq:E1-3}, for which convergence is again naive.
  Thus no convergence issues arise, and
  we conclude that \(H^{*}(\Phi)\) is an isomorphism.
\end{proof}

\begin{proposition}[Universal coefficient theorem]
  \label{thm:UCT-alg}
  Let \(A\) be a bounded \(\Gstar\)-module.
  There is a spectral sequence, natural in~\(A\), with
  \begin{equation}
    E_{2} = \Ext_{\RG}(\HG^{*}(A),\RG) \;\Rightarrow\; \HG^{*}(A^{\vee}).
  \end{equation}
\end{proposition}

\begin{proof}
  By Lemma~\ref{thm:fg-free},
  \(\CarG^{*}(A)\) is \(\RG\)-homotopy equivalent
  to a dg~\(\RG\)-module which is free over~\(\RG\)
  on generators of bounded degrees.
  For such dg~\(\RG\)-modules the proof of~\cite[Prop.~3.5]{AlldayFranzPuppe:orbits1}
  carries over and establishes a spectral sequence with the given \(E_{2}\)~page
  and converging to the cohomology of~\(\Hom_{\RG}(\CarG^{*}(A),\RG)\),
  which is equal to~\(\HG^{*}(A^{\vee})\) by Lemma~\ref{thm:quiso-Car-Hom}.
\end{proof}

\section{Equivariant homology and cohomology}
\label{sec:equiv-cohom-hom}

\subsection{Equivariant de~Rham cohomology}

Although we are primarily interested in manifolds,
singular spaces will inevitably come up in our discussion of the orbit filtration.
Because the Cartan model for equivariant cohomology has technical advantages,
we discuss how to extend it to singular spaces embedded in manifolds and, more generally, in orbifolds.

\subsubsection{Closed supports}

Let \(Y\subset X\) be closed subsets of an orbifold~\(Z\).
We write \(H^{*}(X,Y)\) for the Alexander--Spanier cohomology of the pair~\((X,Y)\) with closed supports.
Our starting point for a de~Rham model is the tautness property of Alexander--Spanier cohomology,
\begin{equation}
  \label{eq:tautness-HXY}
  H^{*}(X,Y) = \dirlim H^{*}(U,V),
\end{equation}
where the direct limit is taken over all open neighbourhood pairs~\((U,V)\) of~\((X,Y)\),
see~\cite[Cor.~6.6.3]{Spanier:1966}.
Since \(U\) and \(V\) are orbifolds, \(H^{*}(U,V)\) may be interpreted as singular cohomology
or, if \(V=\emptyset\), as de~Rham cohomology.

Because direct limits are exact functors, we have \(H^{*}(X)=H^{*}(\Omega_{Z}^{*}(X))\),
where the dg~algebra~\(\Omega_{Z}^{*}(X)\) of germs at~\(X\)
of differential forms on~\(Z\) is defined as
\begin{equation}
  \Omega_{Z}^{*}(X) = \dirlim \Omega^{*}(U) \,;
\end{equation}
the direct limit ranges over all open neighbourhoods~\(U\) of~\(X\) in~\(Z\);
\(\Omega_{Z}^{*}(Y)\) is defined analogously.
Let \(\Omega_{Z}^{*}(X,Y)\) be the kernel
of the restriction map~\(\Omega^{*}_{Z}(X)\to\Omega^{*}_{Z}(Y)\);
it is a dg~module over~\(\Omega^{*}_{Z}(X)\).

\begin{lemma}
  \label{thm:iso-as-dr}
  As a graded vector space, \(H^{*}(\Omega_{Z}^{*}(X,Y))\) is
  naturally isomorphic to \(H^{*}(X,Y)\),
  the Alexander--Spanier cohomology of the pair~\((X,Y)\).
  This isomorphism is compatible with the module structure
  over~\(H^{*}(\Omega_{Z}^{*}(X))=H^{*}(X)\).
\end{lemma}

\begin{proof}
  \def\Hs{H_{\infty}}
  \def\Ss{S_{\infty}}
  The Alexander--Spanier cohomology of an open pair~\((U,V)\) in~\(Z\)
  is naturally isomorphic to the cohomology~\(\Hs^{*}(U,V)\)
  of the complex~\(\Ss^{*}(U,V)\) of smooth singular cochains.
  (For the manifold case see~\cite[Sec.~5.32]{Warner:1983}, for instance.
  It follows for orbifolds by looking at uniformizing charts.)
  We therefore have
  \begin{equation}
    \label{eq:iso-AS-sing-taut}
    H^{*}(X,Y) = \dirlim \Hs^{*}(U,V) = H^{*}\bigl(\dirlim \Ss^{*}(U,V)\bigr).
  \end{equation}
  From the de~Rham theorem, 
  we obtain a natural quasi-isomorphism
  \begin{equation}
    \Omega_{Z}(X) = \dirlim \Omega^{*}(U) \to \dirlim \Ss^{*}(U).
  \end{equation}
  Naturality yields the left vertical arrow in the commutative diagram
  \begin{equation}
    \label{eq:diag-Omega-S}
    \begin{tikzcd}
      0 \arrow{r} & \Omega_{Z}^{*}(X,Y) \arrow{d} \arrow{r} & \Omega_{Z}^{*}(X) \arrow{d} \arrow{r} & \Omega_{Z}^{*}(Y) \arrow{d} \arrow{r} & 0 \\
      0 \arrow{r} & \dirlim \Ss^{*}(U,V) \arrow{r} & \dirlim \Ss^{*}(U) \arrow{r} & \dirlim \Ss^{*}(V) \arrow{r} & 0 \mathrlap{.}
    \end{tikzcd}
  \end{equation}
  The top row is exact:
  Given a differential form~\(\alpha\) on~\(V\supset Y\), choose a smooth function~\(f\)
  on~\(Z\) with support in~\(V\)
  and identically equal to~\(1\) in a smaller neighbourhood~\(V'\) of~\(Y\).
  Then \(f\alpha\) is defined on all of~\(Z\) and restricts to the same germ at~\(Y\) as~\(\alpha\).

  Since the bottom row of~\eqref{eq:diag-Omega-S} is also exact, we conclude from the five lemma
  that the left vertical arrow is a quasi-isomorphism.
  Together with~\eqref{eq:iso-AS-sing-taut} this establishes the claimed isomorphism
  of graded vector spaces.

  That the module structures coincide follows from the fact that
  the isomorphisms induced by the vertical arrows in~\eqref{eq:diag-Omega-S}
  are compatible with (sheaf-theoretic) cup products,
  \cf~\cite[Th.~5.45]{Warner:1983} and~\cite[Sec.~II.7, Thm.~III.3.1]{Bredon:1997}.
\end{proof}

Now assume that \(Z\) is a \(G\)-orbifold and the pair~\((X,Y)\) \(G\)-stable.
We can always assume germs of differential forms at~\(X\) to be defined on \(G\)-stable open neighbourhoods,
which are cofinal among all open neighbourhoods of~\(X\). Hence \(\Omega_{Z}^{*}(X)\) is
a \(\Gstar\)-module, and the same applies to~\(\Omega_{Z}^{*}(Y)\) and \(\Omega_{Z}^{*}(X,Y)\).
So we can consider their Cartan models.
(See~\cite[p.~245]{Meinrenken:1998} for the Cartan model for orbifolds.)
We compare the cohomology of~\(\Omega_{Z}^{*}(X,Y)\) to~\(H_{G}^{*}(X,Y)\),
the Borel \(G\)-equivariant Alexander--Spanier cohomology of the pair~\((X,Y)\).

\begin{lemma}
  \label{thm:iso-as-dr-equiv}
  As \(\RG\)-module,
  \(\HG^{*}(\Omega_{Z}^{*}(X,Y))\) is naturally isomorphic to \(H_{G}^{*}(X,Y)\).
\end{lemma}

In particular, \(\RG\) is isomorphic to~\(H^{*}(BG)\),
the cohomology algebra of the classifying space of~\(G\).

\begin{proof}
  Recall the natural isomorphism
  \begin{equation}
    \HG^{n}(X,Y) = H^{n}(X\times_{G} EG_{k},Y\times_{G} EG_{k})
  \end{equation}
  where \((EG_{k})\) is a family of
  compact free \(G\)-manifolds approximating~\(EG\)
  and \(k\) is chosen sufficiently large compared to~\(n\).
  By Lemma~\ref{thm:iso-as-dr}, this implies
  \begin{equation}
    \HG^{n}(X,Y) = H^{n}(\Omega_{Z\times_{G}EG_{k}}(X\times_{G} EG_{k},Y\times_{G} EG_{k}))
  \end{equation}
  for~\(k\gg n\).
  Also by Lemma~\ref{thm:iso-as-dr} and naturality,
  this isomorphism is compatible with the module structures
  over~\(\HG^{*}({\rm pt})=H^{*}(BG)\).
  
  It therefore suffices to find a family of natural morphisms of dg~algebras
  \begin{equation}
    \CarG(\Omega_{Z}^{*}(X,Y)) \to \Omega^{*}_{Z\times_{G}EG_{k}}(X\times_{G} EG_{k},Y\times_{G} EG_{k})
  \end{equation}
  that induce isomorphisms in any given cohomological degree~\(n\ll k\).
  
  For the case~\((X,Y)=(Z,\emptyset)\),
  this can be found in~\cite[Sec.~2.5]{GuilleminSternberg:1999},
  modulo the natural isomorphism between the  Cartan model
  and the  Weil model for equivariant cohomology,
  \cf~\cite[Thm.~4.2.1]{GuilleminSternberg:1999}.
  This implies in particular~\(\RG=H^{*}(BG)\).
  Noting that due to the compactness of~\(EG_{k}\)
  neighbourhoods of the form~\(U\times_{G} EG_{k}\)
  are cofinal among all neighbourhoods of~\(X\times_{G} EG_{k}\),
  we get the absolute case for arbitrary~\(X\) by taking direct limits
  as in the proof of Lemma~\ref{thm:iso-as-dr}.
  The relative case again follows by the five lemma.
\end{proof}

\begin{remark}
  \label{rem:HG-fg}
  Lemma~\ref{thm:Car-serre} implies
  that \(\HG^{*}(X,Y)\) is finitely generated over~\(\RG\)
  if \(H^{*}(X,Y)\) is finite-dimensional.
  The converse is a consequence of the Eilenberg--Moore theorem.  
\end{remark}

\begin{lemma}
  \label{thm:cond-C}
  If \(G\) acts locally freely on~\(X\setminus Y\), then
  \(\Omega_{Z}^{*}(X,Y)\) is a \(\Wstar\)-module,
  and
  \begin{equation*}
    \Omega_{Z}^{*}(X,Y)_{\rm bas} = \Omega_{Z/G}^{*}(X/G,Y/G).
  \end{equation*}
  In particular, \(\HG^{*}(X,Y)=H^{*}(X/G,Y/G)\).
\end{lemma}

If \(G\) acts locally freely on all of~\(Z\), then \(Z/G\) is an orbifold
and (\(X/G,Y/G)\) is a closed pair in~\(Z/G\), so that \(\Omega_{Z/G}^{*}(X/G,Y/G)\)
has been defined above.
The definition of~\(\Omega_{Z/G}^{*}(X/G,Y/G)\) in the general case
will be given in the proof.

\begin{proof}
  Assume first that \(G\) acts locally freely on~\(Z\), and choose a connection form for~\(Z\).
  Then the de~Rham complex~\(\Omega^{*}(U)\) of any \(G\)-stable open subset~\(U\subset Z\)
  is naturally a \(\Wstar\)-module,
  hence so are \(\Omega_{Z}^{*}(X)\) and~\(\Omega_{Z}^{*}(Y)\),
  and the restriction map~\(\Omega_{Z}^{*}(X)\to\Omega_{Z}^{*}(Y)\)
  is a morphism of \(\Wstar\)-module.
  Thus, its kernel~\(\Omega_{Z}^{*}(X,Y)\) is also a \(\Wstar\)-module.

  We have \(\Omega^{*}(U)_{\rm bas}=\Omega^{*}(U/G)\),
  \cf~\cite[Sec.~2.3.5]{GuilleminSternberg:1999},
  hence \(\Omega_{Z}^{*}(X)_{\rm bas}=\Omega_{Z/G}^{*}(X/G)\).
  The same holds again for~\(Y\), and by naturality we find
  \begin{equation}
    \Omega_{Z}^{*}(X,Y)_{\rm bas}=\Omega_{Z/G}^{*}(X/G,Y/G).
  \end{equation}
  This together with Lemma~\ref{thm:retract-CarGA-Abas} gives
  \begin{equation}
    \HG^{*}(X,Y)=H^{*}(X/G,Y/G).
  \end{equation}

  For the general case we
  observe that \(X\) is metrizable, hence normal.
  We can therefore write
  \begin{align}
    \Omega_{Z}^{*}(X,Y) &= \dirlim \Omega_{Z}^{*}(X,B) \\
    \shortintertext{where the direct limit is taken over all \(G\)-stable closed neighbourhoods~\(B\) of~\(Y\) in~\(X\),}
    \label{eq:dirlim-XYBY}
    &= \dirlim \Omega_{Z\setminus Y}^{*}(X\setminus Y,B\setminus Y)
  \end{align}
  because
  any germ of a differential form on~\(X\setminus Y\) that vanishes on~\(B\setminus Y\)
  can be extended to a germ on~\(X\) vanishing on~\(Y\):
  Assume that the differential form~\(\alpha\) is defined
  on the neighbourhood~\(U\) of~\(X\setminus Y\) in~\(Z\setminus Y\)
  and vanishes on the neighbourhood~\(V\subset U\) of~\(B\setminus Y\).
  Choose disjoint neighbourhoods \(U'\subset U\) of~\(X\setminus V\) and \(V'\) of~\(Y\) in~\(Z\),
  and let \(\beta\) be the form on~\(U'\cup V\cup V'\) that agrees with~\(\alpha\) on~\(U'\cup V\)
  and vanishes on~\(V'\). Then \(\alpha=\beta\) as germs on~\(X\setminus Y\).

  Because \(G\) acts locally freely on~\(X\setminus Y\), it does so on some
  \(G\)-stable open neighbourhood~\(W\subset Z\setminus Y\), and we can write
  \begin{equation}
    \Omega_{Z}^{*}(X,Y) = \dirlim \Omega_{W}^{*}(X\setminus Y,B\setminus Y).
  \end{equation}
  By what we have said above, this is a direct limit of \(\Wstar\)-modules,
  hence a \(\Wstar\)-module itself, and
  \begin{equation}
    \Omega_{Z}^{*}(X,Y)_{\rm bas} = \dirlim \Omega_{W/G}^{*}(X/G\setminus Y/G,B'\setminus Y/G),
  \end{equation}
  where \(B'\) runs through the closed neighbourhoods of~\(Y/G\) in~\(X/G\).
  The right-hand side is our definition of~\(\Omega_{Z/G}^{*}(X/G,Y/G)\)
  (which is independent of the choice of~\(W\)).
  
  Moreover,
  using Lemma~\ref{thm:retract-CarGA-Abas} and arguments similar to the ones above,
  we find
  \begin{align}
    \HG^{*}(\Omega_{Z}^{*}(X,Y)) &= H^{*}(\Omega_{Z}^{*}(X,Y)_{\rm bas})
    = \dirlim H^{*}(X/G\setminus Y/G,B'\setminus Y/G) \\
    &= \dirlim H^{*}(X/G,B') = H^{*}(X/G,Y/G). \qedhere
  \end{align}
\end{proof}

\subsubsection{Compact supports}

We write \(\Hc^{*}(X,Y)\) for the Alexander--Spanier cohomology of the pair~\((X,Y)\)
with compact supports.
There is a natural isomorphism
\begin{equation}
  \label{eq:Hc-tautness}
  \Hc^{*}(X,Y) = \dirlim H^{*}(X,B)
\end{equation}
where the direct limit ranges over all neighbourhoods~\(B\) of~\(Y\)
such that \(X\setminus B\) has compact closure in~\(X\), see~\cite[Thm.~6.6.16]{Spanier:1966}.
Since \(X\) is metrizable, it is enough to consider closed neighbourhoods of~\(B\). 
We therefore set
\begin{equation}
  \Omega_{Z,c}^{*}(X,Y) = \dirlim \Omega_{Z}^{*}(X,B)
\end{equation}
where \(B\) ranges over all closed neighbourhoods of~\(Y\) such that \(X\setminus B\) has compact closure.
Then everything said about closed supports carries over to compact supports.
In particular:

\begin{lemma}
  \label{thm:iso-as-dr-c}
  \(H^{*}(\Omega_{Z,c}^{*}(X,Y))=\Hc^{*}(X,Y)\), the Alexander--Spanier cohomology
  of the pair~\((X,Y)\) with compact supports.
\end{lemma}

\begin{lemma}
  \label{thm:iso-as-dr-equiv-c}
  \(\HG^{*}(\Omega_{Z,c}^{*}(X,Y))=\HGc^{*}(X,Y)\), the \(G\)-equivariant Alexander--Spa\-nier cohomology
  of the pair~\((X,Y)\) with compact supports.
\end{lemma}

\begin{lemma}
  \label{thm:cond-C-compact}
  If \(G\) acts locally freely on~\(X\setminus Y\), then
  \(\Omega_{Z,c}^{*}(X,Y)\) is a \(\Wstar\)-module,
  and
  \begin{equation*}
    \Omega_{Z,c}^{*}(X,Y)_{\rm bas} = \Omega_{Z/G,c}^{*}(X/G,Y/G).
  \end{equation*}
  In particular, \(\HGc^{*}(X,Y)=\Hc^{*}(X/G,Y/G)\).
\end{lemma}

Restriction of forms and extension by~\(0\) give an isomorphism of \(\Gstar\)-modules
\begin{equation}
  \label{eq:OmegacXY-quasi-iso}
  \Omega_{Z,c}^{*}(X,Y)=\Omega_{Z\setminus Y,c}^{*}(X\setminus Y),
\end{equation}
which confirms that cohomology with compact supports is a `single space theory'.
From it we recover the natural isomorphism (see~\cite[Thm.~6.6.11]{Spanier:1966})
\begin{equation}
  \label{eq:Hc-rel-compl}
  \Hc^{*}(X,Y)=\Hc^{*}(X\setminus Y)
\end{equation}
and, using Lemma~\ref{thm:Car-serre},
its equivariant analogue,
\begin{equation}
  \label{eq:HGc-rel-compl}
  \HGc^{*}(X,Y)=\HGc^{*}(X\setminus Y).
\end{equation}

The special case~\(X=Z\) of~\eqref{eq:OmegacXY-quasi-iso}
together with the exact top row of diagram~\eqref{eq:diag-Omega-S}
gives a natural isomorphism
\begin{equation}
  \Omega_{Z,c}^{*}(Y) = \Omega_{c}^{*}(Z) \bigm/ \Omega_{Z,c}^{*}(Z,Y)
  = \Omega_{c}^{*}(Z) \bigm/ \Omega_{c}^{*}(Z\setminus Y),
\end{equation}
which could be used as the definition of~\(\Omega_{Z,c}^{*}(Y)\);
again via~\eqref{eq:OmegacXY-quasi-iso} it can be extended to pairs.
This is the approach taken for example by Guillemin--Sternberg \cite[Sec.~11.1]{GuilleminSternberg:1999}
and also by De\,Concini--Procesi--Vergne~\cite[Sec.~1, App.~A]{DeConciniProcesiVergne:2013},
who prove Lemmas~\ref{thm:iso-as-dr-c} and~\ref{thm:iso-as-dr-equiv-c}
under the assumption that \(X\setminus Y\) is locally contractible
\cite[Prop.~A.4~\&~A.8]{DeConciniProcesiVergne:2013}.

\subsection{Equivariant homology}

We want to define equivariant homology by appropriately dualizing the Cartan model
of differential forms.
The most conceptual approach would be to use currents, as done
by Metzler~\cite[p.~169]{GuilleminSternberg:1999}
and Meinrenken~\cite[Sec.~4]{Meinrenken:2005}.
Because we are ultimately only interested in spaces with
finite-dimensional cohomology, we will work with the algebraic
instead of the topological dual of the de~Rham complex.

For ease of notation, we only write out results for cohomology with closed supports
and homology with compact supports in this section. 
All results remain valid for (co)homology with the other pair of supports.

Let \(Y\subset X\) be closed \(G\)-stable subsets of an orbifold~\(Z\). We say in this case
that \((X,Y)\) is a \emph{\(G\)-pair} in~\(Z\). Often we do not mention the ambient orbifold~\(Z\) explicitly;
we say that \(X\) and~\(Y\) are \emph{\(G\)-spaces}, and we write \(\Omega^{*}(X,Y)\) instead of~\(\Omega_{Z}^{*}(X,Y)\).

\begin{assumption}
  \label{ass:finite-hom}
  From now on,
  whenever we consider the homology or equivariant homology
  (with compact supports) of a \(G\)-pair~\((X,Y)\), we assume that \(H_{*}(X,Y)\)
  or, equivalently, \(H^{*}(X,Y)\) is finite-dimensional. In this case
  there is a natural isomorphism
  \begin{equation}
    H_{*}(X,Y) = H^{*}(X,Y)^{\vee}.
  \end{equation}
  We make no such assumption for \(\hHc_{*}(X,Y)\), the homology of~\((X,Y)\) with closed supports,
  and in fact
  \begin{equation}
    \hHc_{*}(X,Y) = \Hc^{*}(X,Y)^{\vee}
  \end{equation}
  is always true.
\end{assumption}

We define the \emph{equivariant homology}~\(\hHG_{*}(X,Y)\)
of the \(G\)-pair~\((X,Y)\)
as the cohomology of the Cartan model 
\begin{equation}
  \hCG_{*}(X,Y) = \CarG^{*}(\Omega^{*}(X,Y)^{\vee}),
\end{equation}
where \(\Omega^{*}(X,Y)^{\vee}\) denotes the dual
of the \(\Gstar\)-module~\(\Omega^{*}(X,Y)\)
as defined in Section~\ref{sec:cartan-models}.\footnote{%
  Using Propositions~\ref{thm:hHGc-serre} and~\ref{thm:iso-hHGX-hHXG},
  one can show that \(\hHG_{*}(X)\) is isomorphic to
  the equivariant Borel--Moore homology~\(H^{G}_{\textrm{BM},*}(X)\)
  as introduced by Edidin and Graham~\cite[Sec.~2.8]{EdidinGraham:1998}.}

\begin{remark}
  Assume that \(G=T\) is a torus.
  As remarked after~\eqref{eq:map-Car-Hom},
  the canonical map
  \begin{equation}
    \label{eq:map-Car-Hom-T}
    \hCT_{*}(X,Y) \to \Hom_{\RT}(\CT^{*}(X,Y),\RT)
  \end{equation}
  is an isomorphism of dg~\(\RT\)-modules in this case.
  We thus recover the definition of torus-equivariant homology
  given in~\cite{AlldayFranzPuppe:orbits1} and~\cite{AlldayFranzPuppe:orbits4},
  up to the substitution of the Cartan model for the ``singular Cartan model''
  used there.
  For general~\(G\), the analogue of~\eqref{eq:map-Car-Hom-T} still is a quasi-isomorphism
  by Lemma~\ref{thm:quiso-Car-Hom}.
\end{remark}

The cohomology of~\(\Omega^{*}(X,Y)^{\vee}\) is \(H^{*}(X,Y)^{\vee}=H_{*}(X,Y)\)
by Assumption~\ref{ass:finite-hom}.
As a special case of Lemma~\ref{thm:Car-serre} we therefore get:

\begin{proposition}
  \label{thm:hHGc-serre}
  There is a first-quadrant spectral sequence, natural in~\((X,Y)\), with
  \begin{equation*}
    E_{1} = \RG\otimes H_{*}(X,Y) \;\Rightarrow\; \hHG_{*}(X,Y).
  \end{equation*}
\end{proposition}

\begin{proposition}[Universal coefficient theorem]
  \label{thm:UCT}
  There is a spectral sequence, natural in~\((X,Y)\), with
  \begin{align*}
    E_{2} &= \Hom_{\RG}(\HG^{*}(X,Y),\RG) \;\Rightarrow\; \hHG_{*}(X,Y). \\
  \shortintertext{If \(H_{*}(X,Y)\) is finite-dimensional, then
    there is another natural spectral sequence with}
    E_{2} &= \Hom_{\RG}(\hHG_{*}(X,Y),\RG) \;\Rightarrow\; \HG^{*}(X,Y).
  \end{align*}
\end{proposition}

\begin{proof}
  The first spectral sequence is a special case of the algebraic
  universal coefficient theorem (Proposition~\ref{thm:UCT-alg}).
  The same result also gives a spectral sequence of the second form
  converging to~\(\HG^{*}(\Omega^{*}(X,Y)^{\vee\vee})\).
  The latter is equal to~\(\HG^{*}(X,Y)\) by Lemma~\ref{thm:Car-serre}
  because the canonical map~\(\Omega^{*}(X,Y)\to\Omega^{*}(X,Y)^{\vee\vee}\)
  is a quasi-isomorphism by assumption.
\end{proof}

\begin{proposition}
  \label{thm:iso-hHGX-hHXG}
  If \(G\) acts locally freely on~\(X\setminus Y\),
  then there is an isomorphism of \(\RG\)-modules
  \begin{equation*}
    \hHG_{*}(X,Y) \cong H_{*+\dim G}(X/G,Y/G),
  \end{equation*}
  where the \(\RG\)-module structure of~\(H_{*}(X/G,Y/G)\) is
  dual to the one in cohomology.
\end{proposition}

\begin{proof}
  By Lemma~\ref{thm:cond-C} (or Lemma~\ref{thm:cond-C-compact}, depending on supports)
  we have
  \begin{align}
    \Omega^{*}(X,Y)_{\rm bas} &= \Omega^{*}(X/G,Y/G),
  \shortintertext{hence}
    (\Omega^{*}(X,Y)^{\vee})_{\rm bas} &\cong \Omega^{*}(X/G,Y/G)^{\vee}[-\dim G]
  \end{align}
  by Lemma~\ref{thm:Adual-Wstar}. We conclude with Lemma~\ref{thm:retract-CarGA-Abas}.
\end{proof}

\subsection{Equivariant Poincaré duality}
\label{sec:PD-equiv}

In this section the supports of the (co)ho\-mol\-ogy groups do matter,
and we carefully distinguish between them.

Let \(Z\) be an orbifold of dimension~\(n\).
We say that \(Z\) is \emph{locally orientable} if it is locally the quotient
of some~\(\R^{n}\) by a finite subgroup of~\(SO(n)\). In this case,
it is a rational homology manifold, and orientable if and only if \(\hHc_{n}(Z)\cong\R\).
Any manifold is a locally orientable orbifold.

Now assume that \(Z\) is a locally orientable \(G\)-orbifold,
and let \(\pi\colon\tZ\to Z\) be the orientable two-fold cover of~\(Z\)
with a fixed orientation.
If \(Z\) is orientable, then \(\tZ\) consists of two copies of~\(Z\)
with opposite orientations.
The \(G\)-action on~\(Z\) lifts to~\(\tZ\),
\cf~\cite[Lemma~2.13]{AlldayFranzPuppe:orbits4}.
The non-trivial deck transformation~\(\tau\)
commutes with~\(G\) and reverses the orientation.

For any \(G\)-pair~\((X,Y)\) in~\(Z\),
the \((+1)\)-eigenspace of the induced map~\(\tau^{*}\) on \(\HG^{*}(\pi^{-1}(X),\pi^{-1}(Y))\)
is isomorphic to~\(\HG^{*}(X,Y)\) as an \(\RG\)-module.
We define \(\HG^{*}(X,Y;\kktilde)\),
the equivariant cohomology of~\((X,Y)\) 
with twisted coefficients,
to be the \((-1)\)-eigenspace. 
Equivalently, it is the equivariant cohomology of~\(\Omega_{Z}^{*}(X,Y;\kktilde)\),
the \((-1)\)-eigenspace of~\(\tau^{*}\) on~\(\Omega_{Z}^{*}(\pi^{-1}(X),\pi^{-1}(Y))\).
There are analogous definitions for non-equivariant cohomology, for homology and for other supports.

The integration map
\begin{equation}
  \label{eq:int-map}
  I\colon \Cc^{*}(\tZ) \to \R,
  \quad
  \alpha\mapsto \int_{\tZ}\alpha
\end{equation}
(which gives \(0\) if \(|\alpha|<n\))
defines an orientation~\(o=[I]\in \Hc^{n}(Z;\kktilde)^{\vee}=\hHc_{n}(Z;\kktilde)\).
Moreover, it
is a morphism of \(\Gstar\)-modules of degree~\(-n\),
where \(\R\) is considered as a trivial \(\Gstar\)-module.
Hence \(I_{G}=1\otimes I\in\CarG^{*}(\Cc^{*}(\tZ)^{\vee})\) is a cycle 
and defines an equivariant homology class~\(o_{G}\in\hHGc_{n}(Z;\kktilde)\),
which restricts to~\(o\) under the map~\eqref{eq:Car-restriction}.
In fact, it is the unique lift:

\begin{lemma}
  Any orientation~\(o\in\hHc_{n}(Z)\) lifts uniquely
  to an equivariant orientation~\(o_{G}\in\hHGc_{n}(Z)\).
\end{lemma}

\begin{proof}
  The proof of~\cite[Prop.~3.7]{AlldayFranzPuppe:orbits1} carries over.
\end{proof}

The integration map~\(I\) also induces the morphism of \(\Gstar\)-modules
\begin{equation}
  \PD\colon \Omega^{*}(\tZ) \to \Cc^{*}(\tZ)^{\vee},
  \quad
  \alpha\mapsto\Bigl(\beta\mapsto I(\alpha\wedge\beta)\Bigr)
\end{equation}
which interchanges the \((\pm1)\)-eigenspaces.
Based on it,
we obtain equivariant Poin\-caré duality
in the same way as Metzler~\cite[Thm.~10.10.1]{GuilleminSternberg:1999}:

\begin{proposition}
  \label{thm:PD-equiv}
  The map~\(\PD\) induces
  an isomorphism~\(\HG^{*}(Z)\to\hHGc_{n-*}(Z;\kktilde)\) of \(\RG\)-mod\-ules
  of degree~\(-n\).
\end{proposition}

\begin{proof}
  Between the \(E_{1}\)~pages of the spectral sequences from Lemma~\ref{thm:Car-serre},
  the map induced by~\(\PD\) is the \(\RG\)-linear extension of the non-equivariant
  map~\(H^{*}(\PD)\), hence an isomorphism. Thus so is \(\PD\) itself.
\end{proof}

\section{Restriction and induction}
\label{sec:res-ind}

For ease of notation, we again stick to cohomology with closed supports
and homology with compact supports in this section. 
All results remain valid for (co)homology with the other pair of supports
and{\slash}or with twisted coefficients.

Recall that \(T\cong(S^{1})^{r}\) is a maximal torus of~\(G\) and \(W\) the corresponding Weyl group.

\subsection{Restriction}

The restriction~\(\RG\to\RT\) is an isomorphism onto the subalgebra~\((\RT)^{W}\)
of Weyl group invariants.
Moreover, there is a canonical isomorphism of \(\RG\)-modules
\begin{equation}
  \RT = \RG\otimes H^{*}(G/T)
\end{equation}
where \(\RG\) acts only on the first factor of the tensor product.
In particular, \(\RT\) is finitely generated and free over~\(\RG\).
As a consequence, any \(\RT\)-module is finitely generated over~\(\RT\)
if and only if it is finitely generated over~\(\RG\).

Let \((X,Y)\) be a \(G\)-pair. It is canonically a \(T\)-pair by restriction of the action.
The following proposition is a special case
of general results about restrictions of \(\Gstar\)-modules, 
see~\cite[Sec.~6.8]{GuilleminSternberg:1999}.


\begin{proposition}
  \label{thm:HGc-res-ind}
  There are the following isomorphisms, natural in~\((X,Y)\):
  \begin{enumerate}
  \item
    \label{thm:HGc-res-ind-1}
    As \(\RG\)-modules,
    \begin{equation*}
      \HG^{*}(X,Y) = \HT^{*}(X,Y)^{W}
      \quad\text{and}\quad
      \hHG_{*}(X,Y)=\hHT_{*}(X,Y)^{W}.
    \end{equation*}
  \item
    \label{thm:HGc-res-ind-2}
    As \(\RT\)-modules,
    \begin{equation*}
      \HT^{*}(X,Y)=\RT\otimes_{\RG}\HG^{*}(X,Y)
      \quad\text{and}\quad
      \hHT_{*}(X,Y)=\RT\otimes_{\RG}\hHG_{*}(X,Y).
    \end{equation*}
  \end{enumerate}
\end{proposition}

The cohomological parts are well-known and imply
in particular that \(\HG^{*}(X,Y)\) is finitely generated over~\(\RG\) if and only if
\(\HT^{*}(X,Y)\) is finitely generated over~\(\RT\).
This also follows from the fact that both conditions are equivalent
to \(H^{*}(X,Y)\) being finite-dimensional.

\begin{proposition}
  \label{thm:syz-res}
  Let \((X,Y)\) be a \(G\)-pair such that \(\HG^{*}(X,Y)\) is finitely generated over~\(\RG\),
  and let \(0\le j\le r\). Then
  \(\HG^{*}(X,Y)\) is a \(j\)-th syzygy over~\(\RG\)
  if and only if
  \(\HT^{*}(X,Y)\) is a \(j\)-th syzygy over~\(\RT\).
\end{proposition}

\begin{proof}
  Assume that \(\HG^{*}(X,Y)\) is a \(j\)-th syzygy over~\(\RG\).
  By definition, there is an exact sequence
  \begin{equation}
    0 \to \HG^{*}(X,Y) \to F_{1}\to \dots \to F_{j}
  \end{equation}
  with finitely generated free \(\RG\)-modules~\(F_{i}\).
  Because \(\RT\) is free over~\(\RG\), we obtain the exact sequence
  \begin{equation}
    0 \to \RT\otimes_{\RG}\HG^{*}(X,Y) \to \RT\otimes_{\RG}F_{1}\to \dots \to \RT\otimes_{\RG}F_{j}
  \end{equation}
  with finitely generated free \(\RT\)-modules~\(\RT\otimes_{\RG}F_{i}\). This shows
  that \(\HT^{*}(X,Y)=\RT\otimes_{\RG}\HG^{*}(X,Y)\) is a \(j\)-th syzygy over~\(\RT\).

  Now assume that \(\HT^{*}(X,Y)\) is a \(j\)-th syzygy over~\(\RT\).
  This time there is an exact sequence
  \begin{equation}
    0 \to \HT^{*}(X,Y) \to F_{1}\to \dots \to F_{j}
  \end{equation}
  with finitely generated free \(\RT\)-modules~\(F_{i}\).
  Since the~\(F_{i}\) are also finitely generated and free over~\(\RG\),
  this shows that
  \begin{equation}
    \HT^{*}(X,Y) = \HG^{*}(X,Y)\otimes H^{*}(G/T)
  \end{equation}
  is a \(j\)-th syzygy over~\(\RG\). It now follows from criterion~\eqref{thm:syzygy-2} or~\eqref{thm:syzygy-3}
  of Proposition~\ref{thm:syzygy} that the same holds for~\(\HG^{*}(X,Y)\) itself.
\end{proof}

\subsection{Induction}

Let \((X,Y)\) be a \(T\)-pair,
say contained in the \(T\)-orbifold~\(Z\).
Then
\begin{equation}
  \Xhat=G\times_{T}X
  \quad\text{and}\quad  
  \Yhat=G\times_{T}Y
\end{equation}
are  closed \(G\)-stable subsets
of the \(G\)-orbifold~\(\Zhat=G\times_{T}Z\), hence \((\Xhat,\Yhat)\) is a \(G\)-pair.
There is a canonical inclusion~\(Z\hookrightarrow\Zhat\) sending \(z\) to~\([1,z]\),
equivariant with respect to the inclusion~\(T\hookrightarrow G\).
Also note that any \(\RT\)-module is canonically an \(\RG\)-module via the restriction map~\(\RG\to\RT\).

\begin{lemma}
  \label{thm:iso-HGX-HTY}
  The inclusion of pairs~\((X,Y)\hookrightarrow(\Xhat,\Yhat)\)
  induces an isomorphism of \(\RG\)-modules
  \begin{equation*}
    \HG^{*}(\Xhat,\Yhat) = \HT^{*}(X,Y)
  \end{equation*}
\end{lemma}

\begin{proof}
  See~
  \cite[Thm.~24]{DufloVergne:1993} for a proof using equivariant de~Rham theory.
\end{proof}

In particular,
\(\HG^{*}(\Xhat,\Yhat)\) is finitely generated over~\(\RG\)
if and only if \(\HT^{*}(X,Y)\) is finitely generated over~\(\RT\).

To study the behaviour of syzygies under induction from~\(T\) to~\(G\),
we need the following simple algebraic fact.

\begin{lemma}
  \label{thm:depth-loc}
  Let \(A\subset B\) be an extension of commutative rings, \(\qqq\lhd B\) be a prime ideal
  and \(\ppp=\qqq\cap A\). Then for any \(B\)-module~\(M\),
  \begin{equation*}
    \depth_{B_{\qqq}}M_{\qqq} \ge \depth_{A_{\ppp}}M_{\ppp}.
  \end{equation*}
\end{lemma}

\begin{proof}
  Let \(a_{1}\),~\dots,~\(a_{k}\in\ppp_{\ppp}\subset\qqq_{\qqq}\) be an \(M_{\ppp}\)-regular sequence.
  By induction on~\(k\) we show that this sequence is also \(M_{\qqq}\)-regular.
  The case~\(k=0\) is void.

  Assume
  \begin{equation}
    a_{k}\frac{m_{k}}{s_{k}} = a_{1}\frac{m_{1}}{s_{1}}+\dots+a_{k-1}\frac{m_{k-1}}{s_{k-1}}
  \end{equation}
  for some~\(m_{1}\),~\dots,~\(m_{k}\in M\) and some~\(s_{1}\),~\dots,~\(s_{k}\in B\setminus\qqq\).
  Then, for some~\(s\in B\setminus\qqq\),
  \begin{equation}
    a_{k}(s s_{1},\dots,s_{k-1} m_{k}) \in (a_{1},\dots,a_{k-1})M_{\ppp}.
  \end{equation}
  Since \(a_{1}\),~\dots,~\(a_{k}\) is \(M_{\ppp}\)-regular, this implies
  \(s s_{1}\cdots s_{k-1} m_{k} \in (a_{1},\dots,a_{k-1})M_{\ppp}\) and therefore
  \begin{equation}
    \frac{m_{k}}{s_{k}}\in(a_{1},\dots,a_{k-1})M_{\qqq}.
  \end{equation}
  Hence multiplication by~\(a_{k}\) is injective on~\(M_{\qqq}/(a_{1},\dots,a_{k-1})M_{\qqq}\),
  which means that the sequence is \(M_{\qqq}\)-regular.
\end{proof}

\begin{proposition}
  \label{thm:syz-ind}
  Let \((X,Y)\)~and~\((\Xhat,\Yhat)\) be as before, and assume
  that \(\HT^{*}(X,Y)\) is finitely generated over~\(\RT\). Then
  \(\HG^{*}(\Xhat,\Yhat)\) is a \(j\)-th syzygy over~\(\RG\)
  if and only if
  \(\HT^{*}(X,Y)\) is a \(j\)-th syzygy over~\(\RT\).
\end{proposition}

\begin{proof}
  Assume that \(\HT^{*}(X,Y)\) is a \(j\)-th syzygy over~\(\RT\).
  Then there is an exact sequence
  \begin{equation}
    0 \to \HT^{*}(X,Y) \to F_{1}\to \dots \to F_{j}
  \end{equation}
  with finitely generated free \(\RT\)-modules~\(F_{i}\).
  Since the~\(F_{i}\) are also finitely generated and free over~\(\RG\),
  Lemma~\ref{thm:iso-HGX-HTY} implies that \(\HG^{*}(\Xhat,\Yhat)\)
  is a \(j\)-th syzygy over~\(\RG\).

  Now assume that \(\HG^{*}(\Xhat,\Yhat)\) is a \(j\)-th syzygy over~\(\RG\).
  Let \(\qqq\lhd\RT\) be a prime ideal and set \(\ppp=\qqq\cap\RG\).
  Because \(\RT\supset\RG\) is an integral extension of commutative rings,
  Cohen--Seidenberg's going-up theorem implies
  \begin{equation}
    \dim (\RG)_{\qqq} = \height_{\RG}\qqq = \height_{\RT}\ppp
    = \dim (\RT)_{\ppp} \,;
  \end{equation}
  we also have
  \begin{align}
    \depth \HT^{*}(X,Y)_{\qqq} &\ge \depth \HG^{*}(\Xhat,\Yhat)_{\ppp} \\
  \shortintertext{by combining Lemmas~\ref{thm:iso-HGX-HTY} and~\ref{thm:depth-loc}.
    Using Proposition~\ref{thm:syzygy} and the assumption that \(\HG^{*}(\Xhat,\Yhat)\) is a \(j\)-th syzygy,
    we conclude}
    \depth \HT^{*}(X,Y)_{\qqq} 
    &\ge \min(j,\dim (\RG)_{\ppp}) = \min(j,\dim (\RT)_{\qqq}).
  \end{align}
  Thus, \(\HT^{*}(X,Y)\) is a \(j\)-th syzygy over~\(\RT\).
\end{proof}

\begin{remark}
  The case~\(j=1\) of Proposition~\ref{thm:syz-ind}
  can be shown more easily, see the proof of~\cite[Thm.~C.70]{GuilleminGinzburgKarshon:2002}
  or~\cite[Thm.~3.9]{GoertschesRollenske:2011}:
  If \(f\in\RT\) is non-regular for~\(\HT^{*}(X,Y)=\HG^{*}(\Xhat,\Yhat)\),
  then the product~\(\prod_{w\in W}w\cdot f\) gives another such element in~\((\RT)^{W}=\RG\).
\end{remark}

\begin{example}
  Let \(X\) be a projective toric manifold of dimension~\(2r\) with moment polytope~\(P=X/T\).
  It is well-known that \(\HT^{*}(X)\) is free over~\(\RT\) in this case.
  
  Choose two distinct fixed points~\(x\),~\(y\in X\) and set \(Y=X\setminus\{x,y\}\).
  Let \(F\subset P\) be the smallest face of~\(P\) containing (the images of) \(x\)~and~\(y\).
  As shown in~\cite[Sec.~6.1]{Franz:orbits3}, the equivariant cohomology~\(\HT^{*}(Y)\)
  with closed supports is a syzygy of order exactly~\(\dim F-1\) over~\(\RT\);
  the case \(X=(S^{2})^{r}\) with two diametrically opposite vertices~\(x\) and~\(y\)
  of the \(r\)-cube~\(P\) appeared already in~\cite[Sec.~6.1]{AlldayFranzPuppe:orbits1}.
  
  From Proposition~\ref{thm:syz-ind} we see that the equivariant cohomology with closed supports
  of the induced \(G\)-manifold~\(G\times_{T}Y\)
  is a syzygy of order exactly~\(\dim F-1\) over~\(\RG\).
  In particular, syzygies of any order can appear as the equivariant cohomology
  of \(G\)-manifolds.
\end{example}

\section{The constant rank case}
\label{sec:constant}


Let \(X\) be a \(G\)-space such that all \(G\)-isotropy groups in~\(X\) have the same rank,
say equal to~\(b\in\N\). Let
\begin{equation}
  \label{eq:def-Y}
  Y = \{\, x\in X \mid \rank T_{x} = b \,\}
\end{equation}
be the highest-rank stratum for the \(T\)-action on~\(X\);
it is \(N_{G}(T)\)-stable and, by the slice theorem, closed in~\(X\).
For a subtorus~\(L\subset T\) of rank~\(b\), we define
\begin{align}
  X(L) &= \{\, x\in X \mid \text{\(L\) is conjugate in~\(G\) to a maximal torus of~\(G_{x}\)} \,\}, \\
  Y(L) &= \{\,x\in X \mid \text{\(L\) is a maximal torus of~\(G_{x}\)} \,\} = X^{L}.
\end{align}
Note that \(X\) is the union of all such~\(X(L)\),
and \(Y\) is the disjoint union of all such~\(Y(L)\).
Each~\(Y(L)\) is closed in~\(X\).

\begin{lemma}
  \label{thm:properties-XL}
  \( \)
  \begin{enumerate}
  \item The sets~\(X(L)\) partition \(X\).
  \item Each~\(X(L)\) is open in~\(X\).
  \item Each \(X(L)\) is a union of connected components of~\(X\).
  \end{enumerate}
\end{lemma}

\begin{proof}
  Let \(x\in X(L)\).
  If in addition \(x\in X(L')\), then \(X(L)=X(L')\) because all maximal tori of~\(G_{x}\) are conjugate.
  By the slice theorem, any~\(y\) in a sufficiently small neighbourhood of~\(x\) has an isotropy group
  that is conjugate to a subgroup of~\(G_{x}\). By our assumption on~\(X\), this means that \(G_{y}\)
  contains a maximal torus conjugate to a maximal torus of~\(G_{x}\), whence \(y\in X(L)\).
  The last claim follows from the first two.
\end{proof}

The following result is due to Goertsches--Rollenske;
it appears in the proof of~\cite[Prop.~4.2]{GoertschesRollenske:2011}.

\begin{lemma}
  \label{thm:orbit-YL}
  \(G\cdot y\cap Y(L) = N_{G}(L)\cdot y\) for any~\(y\in Y(L)\).
\end{lemma}

\begin{proof}
  \(\supset\): because \(Y(L)\) is stable under~\(N_{G}(L)\).

  \(\subset\): Assume \(gy\in Y(L)\) for some~\(g\in G\).
  Then \(L\) is a maximal torus of both \(G_{y}\) and \(G_{gy}\).
  In other words, \(L\) and \(g L g^{-1}\) are both maximal tori of~\(G_{gy}\),
  hence there is an \(h\in G_{gy}\) such that \(h^{-1} L h=g L g^{-1}\)
  or \((hg)L(hg)^{-1}=L\). This shows that \(hg\in N_{G}(L)\), and therefore
  \(gy = hgy \in N_{G}(L)\cdot y\).
\end{proof}

Let us write \(K=N_{G}(L)\supset T\).

\begin{lemma}
  \label{thm:NGT-NKT-YL-Y}
  The map \(\displaystyle q\colon N_{G}(T) \times_{N_{K}(T)} Y(L)\to Y\cap X(L)\)
  induced by the \(G\)-action is an \(N_{G}(T)\)-equivariant homeomorphism.
\end{lemma}

\begin{proof}
  Clearly, \(q\) is \(N_{G}(T)\)-equivariant.
  Note that \(Y\cap X(L)\) is the disjoint union of of the~\(Y(L')\)
  where \(L'\) runs through all subtori of~\(T\) which are conjugate in~\(G\) to~\(L\).
  We claim that
  \begin{equation}
    \label{eq:W-L}
    \{\, L'\subset T \mid \text{\(L'\) is conjugate in~\(G\) to~\(L\)} \,\} = N_{G}(T)\cdot L = W\cdot L.
  \end{equation}

  \(\supset\): is trivial.

  \def\Grbg{\Gr_{\mkern2mu b}(\ggg)}
  \(\subset\): Assume \(L'=g L g^{-1}\) or, equivalently, \(\lll'=\Ad_{g}\lll\) for some~\(g\in G\).
  This means that \(\lll'\) lies in the \(G\)-orbit of~\(\lll\) for the induced action
  on the Grassmannian~\(\Grbg\). 
  The set~\(B\) of all \(b\)-dimensional abelian Lie subalgebras of~\(\ggg\)
  is closed in~\(\Grbg\) and \(G\)-stable, and all isotropy groups in~\(B\) have rank~\(r\).
  Because \(\lll\) and~\(\lll'\) are contained in~\(B^{T}\),
  Lemma~\ref{thm:orbit-YL}
  (with \(B\) instead of~\(X\) and \(L=T\)) implies
  that \(\lll'\in N_{G}(T)\cdot\lll\).

  Coming back to the original claim, we observe that
  \(N_{K}(T)=N_{G}(T)\cap K\) is the isotropy group of~\(L\)
  for the conjugation action of~\(N_{G}(T)\) on the subtori of~\(T\).
  Hence \(Y\cap X(L)\) consists, like \(N_{G}(T) \times_{N_{K}(T)} Y(L)\),
  of copies of~\(Y(L)\) indexed by the cosets~\(N_{G}(T)/N_{K}(T)\).
  Thus, \(q\) is a homeomorphism.
\end{proof}

\begin{lemma}
  \label{thm:G-K-YL}
  The map~\(\displaystyle f_{1}\colon G\times_{K}Y(L)\to X(L)\)
  induced by the \(G\)-action is a quasi-isomorphism
  (for cohomology with closed or compact supports
  and possibly twisted coefficients).
\end{lemma}

For compact~\(X\) this is again contained
in the proof of~\cite[Prop.~4.2]{GoertschesRollenske:2011},
and also in the proof of~\cite[Thm.~1.2]{Baird:2013}.

\begin{proof}
  We are going to use the Vietoris--Begle theorem, see~\cite[Thm.~6.9.15]{Spanier:1966}
  for a precise statement for Alexander--Spanier cohomology with closed supports.
  We extend it to relative cohomology by the five lemma
  and then to cohomology with compact supports by~\eqref{eq:Hc-tautness}.
  The case of twisted coefficients follows by looking at the eigenspaces
  of the deck transformation~\(\tau\).

  Because \(G\) is compact, the action map~\(G\times X(L)\to X(L)\) is closed, hence so is \(f_{1}\).
  Surjectivity is clear by construction.
  To apply the Vietoris--Begle theorem, it is therefore enough to show
  that the fibres of~\(f_{1}\) are acyclic.
  By \(G\)-equivariance it suffices to study the fibre over some~\(x\in Y(L)\).

  Assume \(f_{1}([g,y])=x\), so that \(gy=x=hy\) for some~\(h\in K\) by Lemma~\ref{thm:orbit-YL}.
  Then \([g,y]=[gh^{-1},hy]=[gh^{-1},x]\) and \(gh^{-1}\in G_{x}\).
  Hence
  \begin{equation}
    f_{1}^{-1}(x)\cong G_{x}/(K\cap G_{x})=G_{x}/N_{G_{x}}(L),
  \end{equation}
  which is acyclic by~\cite[Lemma~3.2]{GoertschesRollenske:2011}
  given that \(L\) is a maximal torus of~\(G_{x}\).
\end{proof}

\begin{proposition}
  \label{thm:incl-Y-X}
  The inclusion~\(Y\hookrightarrow X\) induces an isomorphism
  of \(\RG\)-modules
  \begin{equation*}
    \HG^{*}(X) = \HT^{*}(Y)^{W}
  \end{equation*}
  (for equivariant cohomology with closed or compact supports
  and possibly twisted coefficients).
\end{proposition}

\begin{proof}
  By Lemma~\ref{thm:properties-XL}
  we may assume \(X=X(L)\) for some subtorus~\(L\subset T\) of rank~\(b\).
  
  Consider the commutative diagram
  \begin{equation}
    \begin{tikzcd}
      G\times_{N_{K}(T)}Y(L) \arrow{d}[left]{f_{3}} \arrow{r}{f_{2}} & G\times_{K}Y(L) \arrow{d}{f_{1}} \\
      G\times_{N_{G}(T)}Y \arrow{r}{f} & X \mathrlap{,}
    \end{tikzcd}
  \end{equation}
  where
  \(f_{2}\) comes from the inclusion~\(N_{K}(T)\hookrightarrow K\), and
  \(f_{3}\) is induced from the \(N_{G}(T)\)-equivariant map~\(q\) defined in Lemma~\ref{thm:NGT-NKT-YL-Y},
  \begin{equation}
    f_{3}\colon G\times_{N_{K}(T)} Y(L) = G\times_{N_{G}(T)}\Bigl( N_{G}(T) \times_{N_{K}(T)} Y(L) \Bigr)
    \stackrel{\!(\operatorname{id},q)}\longrightarrow G\times_{N_{G}(T)} Y.
  \end{equation}
  Like \(q\), the map~\(f_{3}\) is a homeomorphism.
  
  Again by~\cite[Lemma~3.2]{GoertschesRollenske:2011},
  the fibre~\(K/N_{K}(T)\) of the bundle map~\(f_{2}\)
  is acyclic,
  so that \(f_{2}\) is a quasi-isomorphism.
  We finally know from Lemma~\ref{thm:G-K-YL} that \(f_{1}\) is a quasi-isomorphism, too.
  Hence \(f\) is a \(G\)-equivariant quasi-isomorphism. It follows that
  \begin{equation}
    \HG^{*}(X) = \HG^{*}(G\times_{N_{G}(T)}Y) = H_{N_{G}(T)}^{*}(Y) = \HT^{*}(Y)^{W}.
    \qedhere
  \end{equation}
\end{proof}

\section{Cohen--Macaulay filtrations}
\label{sec:filtration}

The results in this section hold for (co)homology with either pair of supports
and{\slash}or with twisted coefficients.
For simplicity we only state them for cohomology with closed supports%
{\slash}homology with compact supports and constant coefficients.

\subsection{Definition and first properties}

Let \(X\) be a \(G\)-space, and let \(\FF\) be a filtration
\begin{equation}
  \emptyset=X_{-1}\subset X_{0}\subset\dots\subset X_{r} = X
\end{equation}
of~\(X\) by \(G\)-stable closed subsets.
We call \(\FF\) \emph{Cohen--Macaulay} if 
\(\hHG_{*}(X_{i},X_{i-1})\) is zero or Cohen--Macaulay of dimension~\(r-i\)
for~\(0\le i\le r\).
Because of the following observation, we could substitute
equivariant cohomology for equivariant homology in the definition.

\begin{lemma}
  \label{thm:HG-hHG-CM}
  Let \((X,Y)\) be a \(G\)-pair and \(d\in\N\). Then \(\hHG_{*}(X,Y)\)
  is Cohen--Macaulay of dimension~\(d\) if and only
  if \(\HG^{*}(X,Y)\) is so.
\end{lemma}

\begin{proof}
  It follows from the universal coefficient theorem (Proposition~\ref{thm:UCT})
  that \(\hHG_{*}(X,Y)\) is zero if and only if \(\HG^{*}(X,Y)\)
  is so. Moreover, if \(\HG^{*}(X,Y)\) is Cohen--Macaulay
  of dimension~\(d\), then the spectral sequence converging
  to~\(\hHG_{*}(X,Y)\) collapses at the \(E_{2}\)~page
  by Proposition~\ref{thm:prop-CM}, and
  \begin{equation}
    \hHG_{*}(X,Y) = \Ext_{\RG}^{r-d}(\HG^{*}(X,Y),\RG)[r-d]
  \end{equation}
  is again Cohen--Macaulay of dimension~\(d\).
  The other direction is analogous.
\end{proof}

We will see in Section~\ref{orbit-filtration} that
Cohen--Macaulay filtrations exist. For the moment,
we record several properties.

\begin{proposition}
  \label{thm:seq-hHG-Xi-Xi1}
  Let \((X_{i})\) be a Cohen--Macaulay filtration of~\(X\).
  For any~\(0\le i\le r\) there is a short exact sequence
  \begin{equation*}
    0\to \hHG_{*}(X_{i})\to \hHG_{*}(X)\to \hHG_{*}(X,X_{i})\to 0.
  \end{equation*}
\end{proposition}

\begin{proof}
  As in~\cite[Prop.~4.3]{AlldayFranzPuppe:orbits4}.
\end{proof}

\begin{proposition}
  Let \((X_{i})\) be a Cohen--Macaulay filtration of~\(X\).
  The associated spectral sequence
  converging to~\(\hHG_{*}(X)\) degenerates at~\(E^{1}_{i}=\hHG_{*+i}(X_{i},X_{i-1})\).
\end{proposition}

\begin{proof}
  As in~\cite[Cor.~4.4]{AlldayFranzPuppe:orbits1}.
\end{proof}

\begin{proposition}
  \label{thm:G-filt-T-filt}
  Let \(\FF\) be a \(G\)-stable filtration of~\(X\).
  Then \(\FF\) is a Cohen--Macaulay filtration for~\(X\), considered as a \(G\)-space,
  if and only if it is Cohen--Macaulay for~\(X\), considered as a \(T\)-space.
\end{proposition}

\begin{proof}
  Write \(\FF=(X_{i})\).
  Proposition~\ref{thm:HGc-res-ind}\,\eqref{thm:HGc-res-ind-2} gives,
  for~\(0\le i\le r\),
  \begin{equation}
    \label{eq:HTXiG}
    \hHT_{*}(X_{i},X_{i-1}) = \RT\otimes_{\RG}\hHG_{*}(X_{i},X_{i-1}),
  \end{equation}
  which as an \(\RG\)-module consists of finitely many copies of~\(\hHG_{*}(X_{i},X_{i-1})\).
  Hence \(\hHG_{*}(X_{i},X_{i-1})\) is zero or Cohen--Macaulay of dimension~\(r-i\) over~\(\RG\)
  if and only if \(\hHT_{*}(X_{i},X_{i-1})\) is so.
  But \(\hHT_{*}(X_{i},X_{i-1})\) is Cohen--Macaulay of dimension~\(r-i\) over~\(\RG\)
  if and only if it is so over~\(\RT\),
  \cf~\cite[Lemma 2.6]{GoertschesRollenske:2011}.
\end{proof}

\subsection{The Atiyah--Bredon sequence}

Let \(\FF=(X_{i})\) be a Cohen--Macaulay filtration of~\(X\).
It gives rise to an \emph{Atiyah--Bredon sequence}
\begin{multline}
  0\longrightarrow \HG^{*}(X) \stackrel{\iota^{*}}\longrightarrow \HG^{*}(X_{0})
  \stackrel{\delta_{0}}\longrightarrow \HG^{*+1}(X_{1},X_{0})\stackrel{\delta_{1}}\longrightarrow \cdots \\
  \stackrel{\delta_{r-2}}\longrightarrow \HG^{*+r-1}(X_{r-1},X_{r-2})
  \stackrel{\delta_{r-1}}\longrightarrow \HG^{*+r}(X_{r},X_{r-1})
  \longrightarrow 0.
\end{multline}
Here \(\iota^{*}\) is induced by the inclusion~\(\iota\colon X_{0}\hookrightarrow X\),
and \(\delta_{i}\) for~\(i\ge0\) is
the connecting homomorphism in the long exact sequences
for the triples~\((X_{i+1},X_{i},X_{i-1})\).
We consider this sequence as a dg~\(\RG\)-module~\(\barAB{G}{*}(\FF)\)
with
\begin{equation}
  \barAB{G}{i}(\FF) = \begin{cases}
    \HG^{*}(X) & \text{if \(i=-1\),} \\
    \HG^{*+i}(X_{i},X_{i-1}) & \text{if \(0\le i\le r\)}
  \end{cases}
\end{equation}
and differentials~\(\delta_{i}\) as above plus~\(\delta_{-1}=\iota^{*}\).
We call \(\barAB{G}{*}(\FF)\) the \emph{augmented Atiyah--Bredon complex}
of~\(\FF\).
The (non-augmented) \emph{Atiyah--Bredon complex}~\(\AB{G}{*}(\FF)\) is obtained
by dropping the leading term~\(\barAB{G}{-1}(\FF)=\HG^{*}(X)\).
It is the \(E_{1}\)~page of the spectral sequence
arising from the filtration~\(\FF\) and converging to~\(\HG^{*}(X)\).

\begin{theorem}
  \label{thm:Ext=HAB}
  The cohomology of the Atiyah--Bredon complex is
  \begin{equation*}
    H^{j}(\AB{G}{*}(\FF)) = \Ext_{\RG}^{j}(\hHG_{*}(X),\RG)
  \end{equation*}
  for~\(j\ge0\).
  Under this isomorphism, the map~\(\HG^{*}(X)\to H^{0}(\AB{G}{*}(\FF))\) corresponds
  to the canonical map~\(\HG^{*}(X)\to\Hom_{\RG}(\hHG_{*}(X),\RG)\).
  In particular, the cohomology of~\(\AB{G}{*}(\FF)\)
  and \(\barAB{G}{*}(\FF)\)
  is independent of the Cohen--Macaulay filtration~\(\FF\).
\end{theorem}

\begin{proof}
  This was proven in~\cite[Thm.~5.1]{AlldayFranzPuppe:orbits1}
  for the orbit filtration of a \(T\)-space;  
  see~\cite[Sec.~5]{AlldayFranzPuppe:orbits4} for an alternative
  argument.
  Both proofs generalize to the present context.
  The independence of the Cohen--Macaulay filtration~\(\FF\)
  was already pointed out in~\cite[Rem.~4.9]{AlldayFranzPuppe:orbits1}.
\end{proof}

\subsection{The orbit filtration}
\label{orbit-filtration}

Let \(X\) be a \(G\)-space such that \(H^{*}(X)\) is finite-di\-men\-sional
and such that only finitely many infinitesimal orbit types occur in~\(X\).
We will see in Section~\ref{sec:finite} that
the second assumption is redundant if \(X\) is a manifold or locally orientable orbifold.

\begin{remark}
  \label{rem:finite-G-T}
  We note that \(X\) has finitely many infinitesimal \(G\)-orbit types
  if and only if it has finitely many infinitesimal \(T\)-orbit types.
  That each infinitesimal \(G\)-orbit type restricts to only
  finitely many infinitesimal \(T\)-orbit types can be seen as follows,
  \cf~\cite[Prop.~VIII.3.14]{Borel:1960}:
  Let \(K\subset G\) be an isotropy group occurring in~\(X\).
  Then \(G/K\) is a compact differentiable \(T\)-manifold, hence has only
  finitely many \(T\)-orbit types. Moreover,
  the infinitesimal \(T\)-orbit types in~\(G/K\) depend
  only on the infinitesimal \(G\)-orbit type determined by~\(K\).

  The converse is due to Mostow, 
  \cf~\cite[Thm.~VII.3.1]{Borel:1960}. Note that the proof given there
  only uses that (in the notation of~\cite[Thm.~VII.2.1]{Borel:1960})
  there are finitely many groups~$(S\cap T)^{0}$ for~$S\in\mathcal{S}$,
  see~\cite[p.~97]{Borel:1960}.
\end{remark}

The \emph{\(G\)-orbit filtration}~\((X_{i,G})\) of~\(X\) is defined by
\begin{equation}
  X_{i,G} = \bigl\{\, x\in X \bigm| \rank G_{x} \ge r-i \,\bigr\}
\end{equation}
for~\(-1\le i\le r\).
Each~\(X_{i,G}\) is \(G\)-stable and, by the slice theorem, closed in~\(X\).
If \(G=T\) is a torus, then the strata \(X_{i,T}\setminus X_{i-1,T}\)
are disjoint unions of fixed point sets of subtori, hence smooth if \(X\) is so.
This may fail for non-commutative~\(G\),
see~\cite[Rem.~3.1]{GoertschesRollenske:2011} for an example.

We need a relative version of Proposition~\ref{thm:incl-Y-X}.

\begin{proposition}
  \label{thm:incl-Y-X-rel}
  Let \(0\le i\le r\).
  The inclusion of pairs~\((X_{i,T},X_{i,T}\cap X_{i-1,G})\hookrightarrow(X_{i,G},X_{i-1,G})\)
  induces an isomorphism of \(\RG\)-modules
  \begin{equation*}
    \HG^{*}(X_{i,G},X_{i-1,G}) = \HT^{*}(X_{i,T},X_{i,T}\cap X_{i-1,G})^{W}.
  \end{equation*}
\end{proposition}


\begin{proof}
  For cohomology with compact supports, the claim
  follows immediately from Proposition~\ref{thm:incl-Y-X}
  and the natural isomorphism~\eqref{eq:Hc-rel-compl}.

  To see that it holds for cohomology with closed supports,
  we observe first that Proposition~\ref{thm:incl-Y-X} generalizes by the five lemma
  to pairs~\((X,U)\) where \(U\subset X\) is open (and all isotropy groups
  in~\(X\) have rank~\(b\)).
  Back to the more general case we are considering,
  we have by the tautness of Alexander--Spanier cohomology
  \begin{align}
    \HG^{*}(X_{i,G},X_{i-1,G}) &= \dirlim \HG^{*}(X_{i,G},U) \\
    \shortintertext{where the direct limit is taken over all \(G\)-stable open neighbourhoods of~\(X_{i-1,G}\) in~\(X_{i,G}\).
    By excision and Proposition~\ref{thm:incl-Y-X} for the pair~\((X_{i,G},U)\),}
    &= \dirlim \HG^{*}(X_{i,G}\setminus X_{i-1,G},U\setminus X_{i-1,G}) \\
    &= \dirlim \bigl(\HT^{*}(X_{i,T}\setminus X_{i-1,G},V\setminus X_{i-1,G})^{W}\bigr), \\
    \shortintertext{this time the direct limit being over all \(N_{G}(T)\)-stable open neighbourhoods~\(V\) of~\(X_{i,T}\cap X_{i-1,G}\) in~\(X_{i,T}\).
    Interchanging the direct limit and taking \(W\)-invariants and
    reversing the previous arguments, we finally arrive at}
    &= \HT^{*}(X_{i,T},X_{i,T}\cap X_{i-1,G})^{W}.
    \qedhere
  \end{align}  
\end{proof}

\begin{lemma}
  \label{thm:HXiXi1-fg}
  The vector spaces~\(H^{*}(X_{i,T},X_{i,T}\cap X_{i-1,G})\)
  and \(H^{*}(X_{i,G},X_{i-1,G})\) are finite-dimensional for~\(0\le i\le r\).
\end{lemma}

\begin{proof}
  By induction we can assume that \(H^{*}(X_{j,G},X_{j-1,G})\) is finite-dimensional for~\(j<i\)
  and hence so is \(H^{*}(X_{i-1,G})\). Thus, \(X_{i-1,G}\) satisfies again our assumptions for~\(X\).
  By~\cite[Prop.~4.1.14]{AlldayPuppe:1993}, both \(H^{*}(X_{i,T})\) and
  \begin{equation}
    H^{*}(X_{i,T}\cap X_{i-1,G}) = H^{*}((X_{i-1,G})_{i,T})
  \end{equation}
  are finite-dimensional and therefore also the relative cohomology~\(H^{*}(X_{i,T},X_{i,T}\cap X_{i-1,G})\).
  Hence \(\HT^{*}(X_{i,T},X_{i,T}\cap X_{i-1,G})\)
  is finitely generated over~\(\RT\) and~\(\RG\), and so is the submodule of \(W\)-invariants over~\(\RG\).
  The claim now follows from Proposition~\ref{thm:incl-Y-X-rel} and Remark~\ref{rem:HG-fg}.
\end{proof}

\begin{proposition}
  \label{thm:HGXi-CM}
  The orbit filtration of~\(X\) is Cohen--Macaulay.
\end{proposition}

\begin{proof}
  The torus case was established in~\cite[Prop.~4.1]{AlldayFranzPuppe:orbits4}.\footnote{%
  The local contractability assumptions made there
  were needed to ensure that singular and  Alexander--Spanier cohomology coincide.
  This is not necessary in our setting, \cf~\cite[Rem.~2.17]{AlldayFranzPuppe:orbits4}.}
  In fact, the proof in~\cite{AlldayFranzPuppe:orbits4} shows the following:
  Let \(Y\subset X_{i,T}\) be a \(T\)-stable closed subset containing \(X_{i-1,T}\)
  and such that \(H^{*}(X_{i,T},Y)\) is finite-dimensional. Then \(\HT^{*}(X_{i,T},Y)\) and
  \(\hHT_{*}(X_{i,T},Y)\) are zero or Cohen--Macaulay of dimension~\(r-i\).

  By Remark~\ref{rem:finite-G-T}, only finitely many infinitesimal \(T\)-orbit types occur in~\(X_{i,T}\),
  and \(H^{*}(X_{i,T},X_{i,T}\cap X_{i-1,G})\) is finite-dimensional by Lemma~\ref{thm:HXiXi1-fg}.
  By what we have said above, this implies
  that the \(\RT\)-module \(\HT^{*}(X_{i,T},X_{i,T}\cap X_{i-1,G})\)
  is zero or Cohen--Macaulay of dimension~\(r-i\).
  By~\cite[Lemma~2.7]{GoertschesRollenske:2011}
  the same holds for the \(\RG\)-submodule of \(W\)-invariants.
  Proposition~\ref{thm:incl-Y-X-rel} now shows
  that also \(\HG^{*}(X_{i},X_{i-1})\) is Cohen--Macaulay of dimension~\(r-i\).
  Hence the filtration is Cohen--Macaulay by Lemma~\ref{thm:HG-hHG-CM}.
\end{proof}

\subsection{Partial exactness}
\label{sec:partial}

Let \(X\) be a \(G\)-space
such that \(H^{*}(X)\) is finite-dimensional
and such that only finitely many infinitesimal orbit types occur in~\(X\).
Moreover, let \(\FF=(X_{i})\) be a Cohen--Macaulay filtration of~\(X\),
for example the orbit filtration.

\begin{theorem}
  \label{thm:partial-exact}
  Let \(0\le j\le r\). Then
  \(\HG^{*}(X)\) is a \(j\)-th syzygy over~\(\RG\)
  if and only if
  \(H^{i}(\barAB{G}{*}(\FF))=0\) for all~\(i\le j-2\),
  \ie, if and only if the part
  \begin{equation*}
    0\to \HG^{*}(X) \to \HG^{*}(X_{0}) \to \dots \to \HG^{*+j-1}(X_{j-1},X_{j-2})
  \end{equation*}
  of the Atiyah--Bredon sequence for~\(\FF\) is exact.
\end{theorem}

\begin{proof}
  The torus case was established in~\cite[Thm.~5.7]{AlldayFranzPuppe:orbits1}
  (or~\cite[Thm.~4.8]{AlldayFranzPuppe:orbits4}). It is stated there
  for the orbit filtration of a \(T\)-space, but we know from Theorem~\ref{thm:Ext=HAB}
  that \(H^{*}(\barAB{T}{*}(\FF))\) is independent of the Cohen--Macaulay filtration~\(\FF\).
  
  To reduce the general case to the torus case, we recall from Proposition~\ref{thm:G-filt-T-filt}
  that \(\FF\) is also a Cohen--Macaulay filtration of~\(X\), considered
  as a \(T\)-space. In addition, we observe that
  \begin{equation}
    \label{eq:HiABGT-HiABG}
    H^{i}(\barAB{T}{*}(\FF)) = \RT\otimes_{\RG}H^{i}(\barAB{G}{*}(\FF))
  \end{equation}
  for~\(0\le i\le r\), which follows from the isomorphism
  \begin{equation}
    \HT^{*}(X_{i},X_{i-1}) = \RT\otimes_{\RG}\HG^{*}(X_{i},X_{i-1})
  \end{equation}
  given by Proposition~\ref{thm:HGc-res-ind}\,\eqref{thm:HGc-res-ind-2}
  and the fact that \(\RT\) is free over~\(\RG\).

  By Proposition~\ref{thm:syz-res}, \(\HG^{*}(X)\) is a \(j\)-th syzygy over~\(\RG\)
  if and only if \(\HT^{*}(X)\) is a \(j\)-th syzygy over~\(\RT\).
  As remarked above,
  this latter condition is equivalent to the vanishing of~\(H^{i}(\barAB{T}{*}(\FF))\) for~\(i\le j-2\),
  which in turn is equivalent to the vanishing
  of~\(H^{i}(\barAB{G}{*}(\FF))\) for~\(i\le j-2\) by~\eqref{eq:HiABGT-HiABG}
  and again the freeness of~\(\RT\) over~\(\RG\).
\end{proof}

\begin{corollary}
  Assume that \(X\) is a compact oriented \(G\)-orbifold.
  The \(\RG\)-module \(\HG^{*}(X)\) is torsion-free (reflexive) if and only if
  equivariant Poincaré pairing
  \begin{equation*}
    \HG^{*}(X) \times \HG^{*}(X) \to \RG,
    \quad
    (\alpha,\beta) \mapsto \int_{X}\alpha\beta 
  \end{equation*}
  is non-degenerate (perfect).
\end{corollary}

This characterization of the non-degeneracy of the equivariant Poincaré pairing
has been given by Guillemin--Ginzburg--Karshon~\cite[Thm.~C.70]{GuilleminGinzburgKarshon:2002}
under the assumption that the maximal-rank stratum~\(X_{0,G}\) is smooth and orientable.

\begin{proof}
  This follows from Theorem~\ref{thm:partial-exact} by observing
  that the map
  \begin{equation}
    \HG^{*}(X) \to \Hom_{\RG}(\HG^{*}(X),\RG)
  \end{equation}
  induced by the equivariant Poincaré pairing is the composition
  of the canonical map~\(\HG^{*}(X) \to \Hom_{\RG}(\hHG_{*}(X),\RG)=H^{0}(\AB{G}{*}(\FF))\)
  with the \(\RG\)-transpose of the Poincaré duality isomorphism~\(\HG^{*}(X)\to\hHG_{*}(X)\)
  from Proposition~\ref{thm:PD-equiv}.

  It can alternatively be deduced from Proposition~\ref{thm:syz-res}
  together with the observation that the \(G\)-equivariant Poincaré pairing
  is non-degenerate or perfect if and only if the \(T\)-equivariant one is so.
  We leave the details to the reader.
\end{proof}

\begin{proposition}
  \label{thm:syzygy-free}
  Assume that \(X\) is a compact orientable \(G\)-orbifold.
  If \(\HG^{*}(X)\) is a syzygy of order~\(\ge r/2\), then it is free over~\(\RG\).
\end{proposition}

\begin{proof}
  This follows from Theorems~\ref{thm:Ext=HAB} and~\ref{thm:partial-exact}
  in the same way as for the torus case~\cite[Prop.~5.12]{AlldayFranzPuppe:orbits1}.
  Using Proposition~\ref{thm:syz-res},
  one could also deduce it directly from the torus case.
\end{proof}

The bound~``\(r/2\)'' in Proposition~\ref{thm:syzygy-free} is optimal for any~\(G\):

\begin{example}
  Let~\(a\),~\(b\ge1\) and \(0\le m\le (r-1)/2\).
  Let \(Y\) be the \emph{big polygon space} defined by
  \begin{alignat}{2}
    \|u_{j}\|^{2} + \|z_{j}\|^{2} &= 1  &\qquad& (1\le j\le r), \\
    u_{1}+\dots+u_{2m+1} &= 0
  \end{alignat}
  where \(u_{1}\),~\dots,~\(u_{r}\in\C^{a}\) and~\(z_{1}\),~\dots,~\(z_{r}\in\C^{b}\).
  The torus~\(T\cong(S^{1})^{r}\) acts on~\(Y\) by rotating the variables~\(z_{j}\).
  In~\cite[Sec.~5]{Franz:maximal} it is shown that \(Y\) is a compact orientable \(T\)-manifold
  and that \(\HT^{*}(Y)\) is a syzygy of order exactly~\(m\) over~\(\RT\).

  The induced \(G\)-manifold~\(X=G\times_{T}Y\) is again compact orientable.
  By Proposition~\ref{thm:syz-ind}, \(\HG^{*}(X)\) is a syzygy of order exactly~\(m\) over~\(\RG\).
  In particular, we see that any syzygy order less than~\(r/2\) can occur among the
  equivariant cohomology modules of compact orientable \(G\)-manifolds.
\end{example}

\section{Infinitesimal orbit types}
\label{sec:finite}

\subsection{Preliminaries on torus actions}
\label{sec:prelim-torus}

Recall that \(T\cong(S^{1})^{r}\) is a torus.
Several properties of \(T\)-equivariant cohomology with compact supports
and homology with closed supports
were established in
~\cite{AlldayFranzPuppe:orbits4} for \(T\)-spaces~\(X\)
with finite-dimensional cohomology and{\slash}or
finitely many infinitesimal orbit types\footnote{%
  The stronger assumption of finitely many \emph{orbit types}
  made in~\cite[Sec.~2.1]{AlldayFranzPuppe:orbits4}
  only serves to allow the use of singular cohomology for the quotient~\(X/T\)
  instead of Alexander--Spanier cohomology.};
  we now drop these restrictions and indicate how to adapt the proofs.

Let \(X\) be a \(T\)-space.

\begin{proposition}
  \label{thm:HTc-Xalpha}
  Let \(K\subset T\) be a subtorus and set \(L=T/K\).
  If \(K\) acts trivially on~\(X\), then
  there are isomorphisms of \(\RR\)-modules
  \begin{align*}
    \HTc^{*}(X) &= \RR\otimes_{\RL}\HLc^{*}(X), &
    \hHTc_{*}(X) &= \RR \otimes_{\RL} \hHLc_{*}(X).
  \end{align*}
\end{proposition}

\begin{proof}
  Choose a splitting~\(T\cong K\times L\), which induces an isomorphism~\(\RR\cong\RK\otimes\RL\).
  Since \(K\) acts trivially on~\(X\), we have
  \begin{equation}
    \CTc^{*}(X) = \RK\otimes\CLc^{*}(X) = \RR\otimes_{\RL}\CLc^{*}(X),
  \end{equation}
  which gives the first isomorphism. The homological case is analogous.
\end{proof}

Let \(X_{\alpha}\), \(\alpha\in A\), be the pieces of the partition
of~\(X\) into infinitesimal \(T\)-orbit types.
For~\(\alpha\in A\) let
\(T_{\alpha}\subset T\) be the common identity component of the~\(T_{x}\) with~\(x\in X_{\alpha}\).
Moreover, let \(\ttt_{\alpha}\) be the Lie algebra of~\(T_{\alpha}\),
and let
\(r_{\alpha}=r-\dim T_{\alpha}\) be the common dimension of the \(T\)-orbits in~\(X_{\alpha}\).

Regarding the orbit filtration~\((X_{p})=(X_{p,T})\) of~\(X\),
we have
\begin{equation}
  \label{eq:def-Xp}
  X_{p} = 
  \bigcup_{\!r_{\alpha}\le p\!} X_{\alpha}
\end{equation}
for~\(-1\le p\le r\). Hence
\(X_{p}\setminus X_{p-1}\) is the disjoint union
of the~\(X_{\alpha}\) with~\(r_{\alpha}=p\).

\begin{lemma}
  \label{thm:E1}
  We have
  \begin{align*}
    \HTc^{*}(X_{p}\setminus X_{p-1})
    &= \bigoplus_{\substack{\alpha\in A\\r_{\alpha}=p}} \HTc^{*}(X_{\alpha}), \\
    \hHTc_{*}(X_{p}\setminus X_{p-1})
    &= \prod_{\substack{\alpha\in A\\r_{\alpha}=p}} \hHTc_{*}(X_{\alpha}).
  \end{align*}
\end{lemma}

\begin{proof}
  The first isomorphism is a consequence of the fact that
  \(\Cc^{*}(X_{p}\setminus X_{p-1})\)
  is the direct sum of the~\(\smash{\Cc^{*}(X_{\alpha})}\).
  Its dual~\(\Cc^{*}(X_{p}\setminus X_{p-1})^{\vee}\)
  therefore is the direct product of the \(\Cc^{*}(X_{\alpha})^{\vee}\),
  which implies the second identity.
\end{proof}

For a multiplicative subset~\(\SS \subset\RR\),
set
\begin{equation}
  A(\SS )=\bigl\{\,\alpha\in A\bigm| \SS \cap\ker(\RR\to\RRa)=\emptyset\,\bigr\},
\end{equation}
and let \(X^{\SS }\) be the union of the~\(X_{\alpha}\) with~\(\alpha\in A(\SS )\).
By the slice theorem, \(X^{\SS }\) is closed in~\(X\).

\begin{proposition}[Localization theorem]
  \label{thm:localization}
  Let \(\SS \subset \RR\) be a multiplicative subset.
  The inclusion~\(X^{\SS }\hookrightarrow X\) induces isomorphisms
  of \(\SS ^{-1}\RR\)-modules
  \begin{align*}
    \SS ^{-1}\HTc^{*}(X) &\to \SS ^{-1}\HTc^{*}(X^{\SS }), \\
    \SS ^{-1}\hHTc_{*}(X^{\SS }) &\to \SS ^{-1}\hHTc_{*}(X).
  \end{align*}
\end{proposition}

Note that we put no restriction on the number of infinitesimal orbit types in~\(X\).

\begin{proof}
  For the cohomological claim we assume first \(X=X_{\alpha}\)
  and \(A(\SS )=\emptyset\), say \(f\in \SS \cap\ker(\RR\to\RRa)\).
  Since \(L=T/T_{\alpha}\) acts locally freely on~\(X_{\alpha}\),
  we have \(H_{L,c}^{*}(X_{\alpha}) = \Hc^{*}(X_{\alpha}/L)\)
  by Lemma~\ref{thm:cond-C-compact}.
  In particular, \(H_{L,c}^{*}(X_{\alpha})\) is bounded.
  Because the kernel of the restriction map~\(\RR\to\RRa\)
  is generated by~\(\lll^{*}\subset\RR\),
  this together with Proposition~\ref{thm:HTc-Xalpha} shows that
  a power of~\(f\) annihilates \(\HTc^{*}(X_{\alpha})\),
  so that \(\SS ^{-1}\HTc^{*}(X_{\alpha})=0\).

  For general~\(X\),
  the orbit filtration gives rise to
  the decreasing filtration~\(\CTc^{*}(X\setminus X_{p-1})\) of~\(\CTc^{*}(X)\),
  hence to a spectral sequence converging to~\(\HTc^{*}(X)\) with
  \begin{align}
    E_{0}^{p}(X) &= \CarT^{*}(\Omega_{X,c}^{*}(X_{p},X_{p-1})) = \CTc^{*}(X_{p},X_{p-1}), \\
    E_{1}^{p}(X) &= \HTc^{*}(X_{p},X_{p-1})
    = \HTc^{*}(X_{p}\setminus X_{p-1}),
  \end{align}
  where we have used the isomorphism~\eqref{eq:HGc-rel-compl}.
  Note that \(\HTc^{*}(X_{p}\setminus X_{p-1})\) is given by Lemma~\ref{thm:E1}. 
  Everything said about~\(X\) applies equally to~\(X^{\SS }\)
  with \(A(\SS )\) taking the role of~\(A\).
  Since \(X^{\SS }\) is closed in~\(X\), the inclusion~\(X^{\SS }\hookrightarrow X\)
  induces a map of spectral sequences \(E_{*}(X)\to E_{*}(X^{\SS })\).
  Localizing \(E_{1}(X)\) at~\(\SS \) eliminates all terms~\(\HTc^{*}(X_{\alpha})\) with \(\alpha\notin A(\SS )\).
  The localized map \(\SS ^{-1}E_{1}(X)\to \SS ^{-1}E_{1}(X^{\SS })\) therefore is an isomorphism,
  hence so is the map~\(\SS ^{-1}\HTc^{*}(X) \to \SS ^{-1}\HTc^{*}(X^{\SS })\).

  The homological claim follows from this and the universal coefficient theorem (Proposition~\ref{thm:UCT})
  as in~\cite[Prop.~2.5]{AlldayFranzPuppe:orbits4}.
\end{proof}

\begin{lemma}
  \label{thm:S-Xp}
  For~\(0\le p\le r\), there is a multiplicative subset~\(\SS \subset\RR\)
  such that \(A(\SS )=\{\,\alpha\in A\mid r_{\alpha}\le p\,\}\).
  For such an~\(\SS \),
  the localization map~\(\hHTc_{*}(X_{q},X_{q-1})\to \SS ^{-1}\hHTc_{*}(X_{q},X_{q-1})\) is injective for~\(q\le p\),
  and
  \(\SS ^{-1}\hHTc_{*}(X_{q},X_{q-1})=0\) for~\(q>p\).
\end{lemma}

\begin{proof}
  Note first that \(T\) has only countably many subtori, so that \(A\) is countable.
  For~\(\alpha\in A\), the dimension of
  the dual~\(\lll_{\alpha}^{*}\)
  of the Lie algebra of~\(L_{\alpha}=T/T_{\alpha}\)
  is~\(r_{\alpha}\). Assume \(r_{\alpha}>p\). Since \(\kk\) is uncountable,
  \begin{equation}
    \lll_{\alpha}^{*} \setminus \bigcup_{r_{\beta}\le p} \lll_{\beta}^{*}
  \end{equation}
  cannot be empty. Pick a~\(t_{\alpha}\) from this set and let \(\SS \subset\RR\) be the multiplicative subset
  generated by all such~\(t_{\alpha}\). Then, by construction, \(A(\SS )\) is of the claimed form.

  From Lemma~\ref{thm:E1} 
  and Proposition~\ref{thm:HTc-Xalpha} we see that
  \(\hHTc_{*}(X_{q},X_{q-1})\) is \(\SS \)-torsion-free for~\(q\le p\), which means
  that the localization map is injective. Moreover,
  Proposition~\ref{thm:localization} gives
  \(\SS ^{-1}\hHTc_{*}(X_{q},X_{q-1})=0\) for~\(q>p\)
  as \((X_{q}\setminus X_{q-1})^{\SS }=\emptyset\).
\end{proof}

The following result generalizes \cite[Cor.~4.4]{AlldayFranzPuppe:orbits4} to our setting.
Note that unlike~\cite{AlldayFranzPuppe:orbits1} and~\cite{AlldayFranzPuppe:orbits4},
the proof below does not use the theory of Cohen--Macaulay modules.

\begin{proposition}
  \label{thm:degenerate-E1}
  The spectral sequence induced by the orbit filtration 
  and converging to~\(\hHTc_{*}(X)\)
  degenerates at the \(E_{1}\)~page, which is given by Lemma~\ref{thm:E1}. 
\end{proposition}

  We could alternatively state this result by saying that
  certain short sequences are exact, \cf~Proposition~\ref{thm:seq-hHG-Xi-Xi1}.

\begin{proof}
  We prove the claim by contradiction.
  Let \(d_{k}\), \(k\ge1\), be the first non-vanishing differential,
  say with non-zero value \(a=d_{k}(b)\in\hHTc_{*}(X_{p},X_{p-1})\).
  Choose an~\(\SS \subset\RR\) as in Lemma~\ref{thm:S-Xp}
  and localize the \(E_{k}\)~page of the spectral sequence at~\(\SS \).
  In the localized spectral sequence the differential vanishes
  as \(\SS ^{-1}\hHTc_{*}(X_{q},X_{q-1})=0\) for all~\(q>p\).
  But \(\hHTc_{*}(X_{p},X_{p-1})\) injects into its localization~\(\SS ^{-1}\hHTc_{*}(X_{p},X_{p-1})\),
  which implies \(a=0\).
\end{proof}

Everything done so far in this section goes through for twisted coefficients
with respect to a  fixed orientation cover~\(\tZ\to Z\)
in case the ambient \(T\)-orbifold~\(Z\) is locally orientable.
In addition, one has the following:

\begin{lemma}
  \label{thm:locally-orientable}
  Let \(X\) be a locally orientable \(T\)-orbifold.
  Then each component~\(Y\) of~\(X^{T}\) is again locally orientable.
  Moreover, the restriction  of the orientation cover for~\(X\) to~\(Y\)
  is the orientation cover for~\(Y\).
\end{lemma}

\begin{proof}
  The first part is well-known, \cf~\cite[Thm.~V.3.2]{Borel:1960}.
  With our tools, it can be seen as follows:
  Choose a uniformizing chart~\(\tilde U\to U\) of~\(X\) at~\(y\in Y\) such that \(\tilde U\) is an open ball with linearized torus action.
  Because \(X\) is locally orientable, \(\Hc^{*}(U)\) is non-zero and concentrated
  in top degree. Thus \(\HTc^{*}(U)=\RT\otimes\Hc^{*}(U)\). By the localization theorem,
  this implies that \(\HTc^{*}(U^{T})=\RT\otimes\Hc^{*}(U^{T})\) is non-zero.
  Hence \(U^{T}\) is the quotient of~\(\tilde U^{T}\)
  by an orientation-preserving action since otherwise we would have \(\Hc^{*}(U^{T})=0\).

  For the the second claim it suffices to observe that \(Y\) is orientable if and only if \(X\) is so.
  This is a consequence of the fact that the normal (orbi)bundle of~\(Y\) in~\(X\) is always orientable.
  See~\cite[Cor.~2]{Duflot:1983} for the case of manifolds, from which the orbifold case follows.
\end{proof}

\subsection{The number of infinitesimal orbit types}

\begin{theorem}
  \label{thm:finite-inf-orbits}
  Let \(X\) be a 
  locally orientable \(G\)-orbifold.
  If \(H^{*}(X)\) is finite-dimensional, then
  only finitely many infinitesimal orbit types occur in~\(X\).
\end{theorem}

\begin{proof}
  By Mostow's argument 
  we may assume that \(G=T\) is a torus, see Remark~\ref{rem:finite-G-T}.

  Assume that there are infinitely many infinitesimal orbit types.
  By Proposition~\ref{thm:degenerate-E1} and the previous section,
  the spectral sequence induced by the orbit filtration
  and converging to~\(\hHTc_{*}(X;\kktilde)\)
  degenerates at the \(E_{1}\)~page.
  Each of the infinitely many terms~\(\hHTc_{*}(X_{\alpha};\kktilde)\)
  appearing in the twisted version of Lemma~\ref{thm:E1} 
  is non-zero
  as it contains an equivariant orientation by Lemma~\ref{thm:locally-orientable}.
  Thus, \(E_{\infty}(X)=E_{1}(X)\) is not finitely generated.
  This implies that \(\hHTc_{*}(X;\kktilde)\) cannot be finitely generated either
  because \(E_{\infty}(X)\) arises from a filtration of it and \(\RR\) is Noetherian.

  On the other hand,
  \(H^{*}(X)\) is finite-dimensional by assumption,
  hence \(\HT^{*}(X)\) is finitely generated over~\(\RT\).
  By equivariant Poincaré duality (Proposition~\ref{thm:PD-equiv}),
  the same holds for~\(\hHTc_{*}(X;\kktilde)\).
  Contradiction.
\end{proof}

We conclude 
with several examples.
The first one, due to Montgomery~\cite[\S 3]{Yang:1957},
illustrates the difference between our result and Mann's~\cite[Thm.~3.5]{Mann:1962}.
Recall that Mann showed that any orientable cohomology manifold
with finitely generated integral cohomology has only finitely many orbit types.

\begin{example}
  Consider a circle acting on~\(\R^{3}\) via rotations about some axis.
  For each~\(m\in\N\), choose an open solid torus
  equivariantly diffeomorphic to~\(A_{m}\times S^{1}\), where \(A_{m}\subset\R^{3}\)
  is an open disc with centre~\(b_{m}\).
  (These tori have to be disjoint.)
  Remove the central circle~\(b_{m}\times S^{1}\) of each torus
  and glue the remaining manifold to copies of the~\(A_{m}\times S^{1}\)
  via the diffeomorphisms~\(f_{m}\)
  of~\((A_{m}\setminus b_{m})\times S^{1}\cong(0,1)\times S^{1}\times S^{1}\)
  given by~\(f_{m}(r,\alpha,\beta)=(r,\alpha\beta^{1-m},\beta^{m}/\alpha)\).
  Each subgroup~\(\Z_{m}\subset S^{1}\) occurs as an isotropy group
  in the orientable manifold~\(X\) thus obtained.
  One can check that \(H_{1}(X;\Z)\) is isomorphic to the direct sum of all~\(\Z_{m}\)
  and that \(H_{2}(X;\Z)\) vanishes. In particular, \(H_{*}(X;\Z)\) is not finitely generated.
  On the other hand, \(H_{*}(X;\R)\)
  is finite-dimensional, and Theorem~\ref{thm:finite-inf-orbits} holds (trivially).
\end{example}

\begin{example}
  We now imitate this idea for the standard action
  of~\(T=S^{1}\times S^{1}\) on~\(\C^{2}\).
  In the free part of the action, choose disjoint open subsets of the form~\(A_{m}\times T\),
  where the slice~\(A_{m}\subset\C^{2}\) is again an open disc with centre~\(b_{m}\).
  Remove \(b_{m}\times T\) from the manifold and glue the rest
  to a copy of~\(A_{m}\times T\) via the diffeomorphism~\(f_{m}\)
  of~\((A_{m}\setminus b_{m})\times T\cong(0,1)\times(S^{1})^{3}\)
  given by~\(f_{m}(r,\alpha,\beta,\gamma)=(r,\alpha\beta,\beta^{m}\gamma,\alpha)\).
  Each such operation increases the Betti sum of the manifold by~\(4\),
  and it produces the isotropy group~\(\{(g,g^{-m})\mid g\in S^{1}\}\subset T\).
  If this is done for all~\(m\in\Z\), we obtain a manifold with infinitely many
  distinct infinitesimal isotropy groups and infinite Betti sum.
\end{example}

We finally recall an example of Kister and Mann~\cite[Sec.~1]{KisterMann:1962}
showing that Theorem~\ref{thm:finite-inf-orbits} cannot be naively extended
to actions on more general spaces.

\begin{example}
  Take countably many closed discs, connected by a line through their centres.
  This space is properly homotopy-equivalent to~\(\R\), hence contractible.
  On the other hand, the discs can be rotated independently via characters~\(T\to S^{1}\)
  to produce infinitely many distinct infinitesimal orbit types.
  Note that an argument as in the proof of Theorem~\ref{thm:finite-inf-orbits}
  fails here
  because the non-fixed points do not contribute to~\(\hHTc_{*}(X)\). 
  In fact, for a closed disc~\(D\) with centre~\(x\) one has \(\hHc_{*}(D,\{x\})=0\),
  hence also \(\hHTc_{*}(D,\{x\})=0\) by Proposition~\ref{thm:hHGc-serre}.
\end{example}

\end{document}